\documentclass[11pt]{amsart}

\usepackage{url}
\usepackage{epsfig,color}
\usepackage{amssymb,amsrefs}
\usepackage{epstopdf}
\usepackage{amsthm}
\usepackage[utf8]{inputenc}
\usepackage[T1]{fontenc}
\usepackage{mathtools}
\usepackage[thinc]{esdiff}
\usepackage[toc,page]{appendix}
\usepackage{marginnote}

\usepackage{nomencl}
\makenomenclature

\makeatletter

\newcommand{\addresseshere}{%
	\enddoc@text\let\enddoc@text\relax
}

\newsavebox\myboxA
\newsavebox\myboxB
\newlength\mylenA

\newcommand*\xoverline[2][0.75]{%
    \sbox{\myboxA}{$\m@th#2$}%
    \setbox\myboxB\null% Phantom box
    \ht\myboxB=\ht\myboxA%
    \dp\myboxB=\dp\myboxA%
    \wd\myboxB=#1\wd\myboxA% Scale phantom
    \sbox\myboxB{$\m@th\overline{\copy\myboxB}$}%  Overlined phantom
    \setlength\mylenA{\the\wd\myboxA}%   calc width diff
    \addtolength\mylenA{-\the\wd\myboxB}%
    \ifdim\wd\myboxB<\wd\myboxA%
       \rlap{\hskip 0.5\mylenA\usebox\myboxB}{\usebox\myboxA}%
    \else
        \hskip -0.5\mylenA\rlap{\usebox\myboxA}{\hskip 0.5\mylenA\usebox\myboxB}%
    \fi}
\makeatother
\newcommand{\ols}[1]{\mskip.5\thinmuskip\overline{\mskip-.5\thinmuskip {#1} \mskip-.5\thinmuskip}\mskip.5\thinmuskip} % overline short
\newcommand{\olsi}[1]{\,\overline{\!{#1}}} % overline short italic
\makeatletter
\newcommand\closure[1]{
  \tctestifnum{\count@stringtoks{#1}>1} %checks if number of chars in arg > 1 (including '\')
  {\ols{#1}} %if arg is longer than just one char, e.g. \mathbb{Q}, \mathbb{F},...
  {\olsi{#1}} %if arg is just one char, e.g. K, L,...
}
\long\def\count@stringtoks#1{\tc@earg\count@toks{\string#1}}
\long\def\count@toks#1{\the\numexpr-1\count@@toks#1.\tc@endcnt}
\long\def\count@@toks#1#2\tc@endcnt{+1\tc@ifempty{#2}{\relax}{\count@@toks#2\tc@endcnt}}
\def\tc@ifempty#1{\tc@testxifx{\expandafter\relax\detokenize{#1}\relax}}
\long\def\tc@earg#1#2{\expandafter#1\expandafter{#2}}
\long\def\tctestifnum#1{\tctestifcon{\ifnum#1\relax}}
\long\def\tctestifcon#1{#1\expandafter\tc@exfirst\else\expandafter\tc@exsecond\fi}
\long\def\tc@testxifx{\tc@earg\tctestifx}
\long\def\tctestifx#1{\tctestifcon{\ifx#1}}
\long\def\tc@exfirst#1#2{#1}
\long\def\tc@exsecond#1#2{#2}
\makeatother

\newtheorem{theorem}{Theorem}
\newtheorem{lemma}[theorem]{Lemma}
\newtheorem{remark}[theorem]{Remark}
\newtheorem{cor}[theorem]{Corollary}
\newtheorem{conj}[theorem]{Conjecture}

\newtheorem{definition}[theorem]{Definition}
\newtheorem{notation}[theorem]{Notation}

\makeatletter
\newtheorem*{rep@theorem}{\rep@title}
\newcommand{\newreptheorem}[2]{%
	\newenvironment{rep#1}[1]{%
		\def\rep@title{#2 \ref{##1}}%
		\begin{rep@theorem}}%
		{\end{rep@theorem}}}
\makeatother

\newreptheorem{theorem}{Theorem}

\newreptheorem{lemma}{Lemma}

\newcommand{\FF}{{\mathbb F}}

\newcommand{\fq}{{\FF_q}}

\newcommand{\ftwo}{{\FF_2}}

\newcommand{\ftwom}{{\FF_{2^m}}}
\newcommand{\ftwomstar}{{\FF^*_{2^m}}}

\newcommand{\Tr}{\operatorname{Tr}}

\begin{document}
	
	%-------------------------------------
	\title{Difference sets and tri-weight linear codes from trinomials over binary fields}

	\author[O. Ahmadi]{Omran Ahmadi}
	\address{Institute for Research in Fundamental Sciences, Iran}
	\email{oahmadid@ipm.ir}
	\author[M. Shafaeiabr]{Masoud Shafaeiabr}
     \address{Institute for Research in Fundamental Sciences, Iran}
     \email{masoudshafaei@ipm.ir}

%########################################################################################
\begin{abstract}
We confirm a conjecture of Cun Sheng Ding~\cite{Ding-Discrete} claiming that the punctured value-sets of a list of eleven trinomials over odd-degree extensions of the binary field give rise to difference sets with Singer parameters. In the course of confirming the conjecture, we show that these trinomials share the remarkable property that every element of the value-set of each trinomial has either one or four preiamges . We also give partial resolution of another conjecture of Cun Sheng Ding~\cite{Ding-Discrete} claiming that linear codes constructed from those eleven trinomials are tri-weight.
\end{abstract}

\maketitle
\let\thefootnote\relax\footnotetext{ \nonumber}
\section{Introduction}

Let $G$ be a cyclic group of order $v$. A subset $\mathcal{D}$ of size $k$ of $G$ is a cyclic difference set with parameter $(v,k,\lambda)$ if for  every non-identity element $g$ of $G$ the number of $(x,y)$ solutions of the equation $g=xy^{-1}$ with $x,y\in \mathcal{D}$ is exactly $\lambda$. 

Difference sets have many applications in the construction of combinatorial designs, error correcting codes, sequences with good correlation properties as well as the study of finite geometries and cryptographic functions.

As it is the case with the construction of many other combinatorial objects, the theory of finite fields is of great importance in the construction of  difference sets. Cyclic difference sets of special interest which are constructed from finite fields are the so called {\it{difference sets with Singer parameters}} for which either  
\[
v=\frac{q^n-1}{q-1},\;\;k=\frac{q^{n-1}-1}{q-1},\;\;\lambda=\frac{q^{n-2}-1}{q-1}
\]
or
\[
v=\frac{q^n-1}{q-1},\;\;k=q^{n-1},\;\;\lambda=q^{n-2}(q-1)
\]
where $q$ is a prime power and $n\ge 3$. 

In this paper, we are interested in the construction of cyclic difference sets and linear codes from polynomials and rational functions defined over the binary finite field $\ftwom$ with $2^m$ elements where $m$ is an odd number. Our main result is the following theorem stated as Conjecture 36 in ~\cite{Ding-Discrete}. 

\iffalse
In order to state our main result, we need the following definitions and notations.

\begin{definition}\label{Value-set}
	Let $\ftwom$ denote the degree-$m$ extension of the binary field $\ftwo$, and $\ftwom[x]$ denote the univariate polynomial ring over $\ftwom$. For any $f\in \ftwom[x]$, the set
	\[D{{(f)}}^*=\{f(x):x\in \mathbb{F}_{{2}^{m}}\}\setminus\{0\},\]is called the punctured value-set of $f$ and its size is denoted by $|D(f)^*|$.
\end{definition}

Here is our main result on difference sets.
\fi
\begin{theorem}\label{Ding-conj-diff}
 Let $m\ge 5$ be an odd number, $\ftwom$ denote the degree-$m$ extension of the binary field $\ftwo$, and $\ftwom[x]$ denote the univariate polynomial ring over $\ftwom$. Furthermore, let $\ftwomstar$ denote the multiplicative subgroup of $\ftwom$, and for a given $f\in\ftwom[x]$, $D(f)^*$ denotes the punctured value-set it, that is,
 	\[D{{(f)}}^*=\{f(x):x\in \mathbb{F}_{{2}^{m}}\}\setminus\{0\}.\]
 
  Then the punctured value-set of any of trinomials appearing below is a difference set in $(\ftwomstar,\times)$ with singer parameters $(2^m-1, 2^{m-1}, 2^{m-2})$:
\begin{itemize}
	\item[(a)] $f_1(x)=x^{2^m-17}+x^{(2^m+19)/3}+x$.
	\item[(b)] $f_2(x)=x^{2^m-2^{m-4}-1}+x^{2^m-(2^{m-2}+4)/3}+x$.
	\item[(c)] $f_3(x)= x^{2^m-3}+x^{2^{(m+3)/2}+2^{(m+1)/2}+4}+x$.
	\item[(d)]$f_4(x)= x^{2^{m}-2^{(m-1)/2}-1}+x^{2^{(m-1)}-2^{(m-1)/2}}+x$.
	\item[(e)]$f_5(x)=x^{2^m-2-(2^{m-1}-2^2)/3}+x^{2^m-2^2-(2^m-2^3)/3}+x$.
	\item[(f)] $f_6(x)=x^{2^m-2^{(m+1)/2}+2^{(m-1)/2}}+x^{2^m-2^{(m+1)/2}-1}+x$.
	\item[(g)]$f_7(x)=x^{2^m-3(2^{(m+1)/2}-1)}+x^{2^{(m+1)/2}+2^{(m-1)/2}-2}+x$.
	\item[(h)]$f_8(x)=x^{2^m-2^{m-2}-1}+x^{2^{m-1}-2}+x$.
	\item[(i)] $f_9(x)= x^{2^{m}-2^{(m+3)/2}-3}+x^{2^{(m+1)/2}+2}+x$.
	\item[(j)]$f_{10}(x)=x^{2^m-3(2^{(m-1)/2}+1)}+x^{2^{m-1}-1}+x$.
	\item[(k)]$f_{11}(x)=x^{2^m-5}+x^6+x$.
\end{itemize}	
	
\end{theorem}

In the course of proving the above theorem, we prove that the polynomials appearing above share the remarkable property that any element of the value-set has 1 or 4 preimages in $\ftwom$. This naturally leads to the partial resolution of another conjecture of Cun~Sheng Ding stated as Conjecture 37~\cite{Ding-Discrete} and appearing as Theorem~\ref{Ding-conj-codes} in this paper.

This paper is organized as follows. We end this section with a list on conventions and notations. In Section~\ref{Prem}, we gather some preliminary results. In Section~\ref{Value-set-equivalence}, we show that the proof of the main result follows if we prove that the punctured value-set of four rational functions give rise to difference set with Singer parameters. In Section~\ref{Main-result-proof} we prove the main result. Section~\ref{Codes} contains our partial resolution of a conjecture of Cun Sheng Ding about tri-weight linear codes constructed from the above polynomials. Finally, Section~\ref{Conclusions} contains our concluding remarks.
%******************************************************************
\subsection{Conventions and Notations}
In this subsection we 
\begin{itemize}
	
\item Throughout the paper $m$ will be an odd integer, and $\sigma=2^{\frac{m+1}{2}}$.	
	
\item $\fq$, $\fq^*$, and ${\xoverline{\fq}}$ denote the finite field with $q$ elements, the multiplicative subgroup of $\fq$, and the algebraic closure of $\fq$, respectively. 	

\item For $x\in\ftwom$, $\Tr(x)$ denotes the absolute trace from $\ftwom$ to $\ftwo$, that is,
\[\Tr(x)=x+x^2+x^{2^2}+\cdots+x^{2^{m-1}}.\]

\item We denote by $T_n$ the set of elements of $\ftwom$ whose $n$-th power is of trace one, that is,
\[T_n=\{x\in GF(2^m)| \Tr (x^n)=1\},\]

\item  We denote the set of elements of trace zero by $T_0$.

\item If $X$ and $Y$ are two sets and $f:X \xrightarrow{} Y$ is a function, then for each $y \in Y$, the set of elements whose image under $f$ is denoted by $P_{f}(y)$, that is, 
	\[P_{f}(y)=\{x \in X|f(x)=y\}.\]
	
\item  If $X$ and $Y$ are two sets, $f:X \xrightarrow{} Y$ is a function and for each $y \in Y$, $|P_{f}(y)|\in\{0,1,a\}$ for some positive integer $a$, then $f$ is called a $(0,1,a)$-map. In this paper, a will be either 3 or 4.
	
\item For a set $S$, $|S|$ denotes the number of elements of $S$.

\item Given a function $f:\ftwomstar\longrightarrow \ftwomstar$, the set
	\[D{{(f)}}^*=\{f(x):x\in \ftwomstar\}\setminus\{0\},\]is called the punctured value-set of $f$. Notice that since all the polynomials appearing in Theorem~\ref{Ding-conj-diff} vanish at zero, $D(f)^*$ obtained using  this new definition for a polynomial $f$ appearing in the theorem coincides with $D(f)^*$ obtained from previous definition given above. Thus, for the rest of the paper, we will work with this new definition.

\item 	If $G$ is a group and $\mathcal{D}\subset G$, then we denote by $\mathcal{D}^k$  the set of $k$-th power of the elements of $\mathcal{D}$, that is,
\[
\mathcal{D}^k=\left\{d^k|d\in\mathcal{D}\right\}.
\]

\item A multi-set is denoted by double braces $\{\{\}\}$.

\item If $C$ is a curve defined over $\ftwom$, then $\#C(\ftwom)$ denotes the number of the affine $\ftwom$-rational points on $C$. 

\end{itemize}

%************************************************************************
\section{Preliminaries}\label{Prem}
In this section we gather and prove some results which will be used in the rest
of the paper.

\subsection{Curves with no finite point over odd degree extensions of $\ftwo$}\label{Curves}
In this paper, by a curve $C$ over the finite field $\fq$ we mean an either reducible or irreducible  algebraic set of dimension one over $\fq$. Furthermore, in this paper we will be interested in, mostly, with the combinatorial properties of the curves such as the number of the points  and will not be interested in geometric properties such as irreducibility and smoothness and so we will not try to prove the irreducibility of them even though many of the algebraic sets that we encounter will be irreducible.

%\subsubsection{Curves with no finite point over odd degree extensions of $\ftwo$}
Now, let $\fq[x,y]$ and ${\xoverline{\fq}}[x,y]$ denote the bivariate polynomial rings over $\fq$ and its algebraic closure, respectively. 
Furthermore, let the map $\psi$ be as follows
\begin{eqnarray}
\psi:&{\xoverline{\fq}}[x,y]&\longrightarrow {\xoverline{\fq}}[x,y]\\\nonumber
&\sum a_ib_j x^iy^j&\longrightarrow \sum (a_ib_j)^qx^iy^j.
\end{eqnarray}
We notice that the map $\psi$ is a ring homomorphism, and also if $f\in \fq[x,y]$, then $\psi(f)=f$. This fact leads to  the following lemma.

%As  in the sequel, we will need some results on the number of $\ftwom$-rational points on some families of %curves, we prove some results related to curves with no finite point over odd degree extensions of %$\ftwo$  in the following lemmas. 

\begin{lemma}\label{QuadraticForm}
Let $f\in \ftwo[x,y]$ be an irreducible polynomial over $\ftwo$ which is reducible over $\FF_{4}$. Then there are nonzero polynomials $g,h\in\ftwo[x,y]$ such that $f=g^2+gh+h^2$.
\end{lemma}
\begin{proof}
Let $f_1\in \FF_4[x,y]$ be a proper divisor of $f$ in $\FF_4[x,y]$. If we let $\FF_4=\ftwo(t)$ where $t^2=t+1$, then there are two polynomials $g$ and $h$ in $\ftwo[x,y]$ such that $f_1=g+th$. Obviously, $h\neq 0$ and the polynomials $g$ and $h$ have no common factor as otherwise $f$ would have a proper divisor in $\ftwo[x]$ which is a contradiction. Furthermore, since the map $\psi$ as defined above is a homomorphism which fixes $\ftwo[x,y]$, we see that $\psi(f_1)=g+t^2h=g+(t+1)h$ is a proper divisor of $f$, too. Now, since $h\neq 0$ and the polynomials $g$ and $h$ have no common factor, $f$ and $\psi(f_1)$ have no common factor. Thus 
$f_1\psi(f_1)=g^2+hg+h^2$ is a divisor of $f$. But $f$ is irreducible in $\ftwo[x,y]$. So, we must have $f=g^2+hg+h^2$ and $g\neq 0$ which finishes the proof.
\end{proof}
\begin{lemma}\label{No-point}
Let $f,g,h\in \ftwo[x,y]$ such that $f=g^2+gh+h^2$. Furthermore, suppose that there is no finite double point on $f(x,y)=0$ over the odd-degree extensions of $\ftwo$, i.e., there is no odd number $m$ and $(a,b)\in {\ftwom}\times {\ftwom}$ such that $f(a,b)=0,  \diffp{f}{x}(a,b)$ and $\diffp{f}{y}(a,b)=0$. Then $f(x,y)=0$ has no solution over the odd-degree extensions of $\ftwo$. 
\end{lemma}
\begin{proof}
Suppose for some odd number $m$, there are $a,b\in\ftwom$ such that $f(a,b)=0$. Then 
\begin{equation}\label{gh-equation}
f=g(a,b)^2+g(a,b)h(a,b)+h(a,b)^2=0.
\end{equation}
From this, we notice that if $h(a,b)=0$, then $g(a,b)=0$ which in turn implies that $ \diffp{f}{x}(a,b)$ and $\diffp{f}{y}(a,b)=0$. But this is a contradiction. So, we must have $h(a,b)\neq 0$ and $g(a,b)\neq 0$. Now, if we let $t=\frac{g(a,b)}{h(a,b)}$ which is in $\ftwom$, then from ~\eqref{gh-equation} we deduce that $$g^2+gh+h^2=h(a,b)^2(t^2+t+1)=0.$$ But this is again a contradiction since $h(a,b)\neq 0$ and $t^2+t+1\neq 0$ whenever $t$ is in the odd-degree extensions of $\ftwo$.  
\end{proof}

The following is an immediate corollary of the above lemmas.
\begin{cor}\label{NoPoint-Cor}
	Let $f\in \ftwo[x,y]$ be an irreducible polynomial over $\ftwo$ which is reducible over $\FF_{4}$. Furthermore, suppose that there is no finite double point on $f(x,y)=0$ over the odd-degree extensions of $\ftwo$. Then $f(x,y)=0$ has no solution over the odd-degree extensions of $\ftwo$. 
\end{cor}
%#########################################################################################

\iffalse
$x^3+x+a$ is irreducible for $\frac{2^m+1}{3}$, one root for $2^{m-1}$ and 3 roots for $\frac{2^{m-1}-1}{3}$.

The following lemma is a very special case of a much more general result of Bluher~\cite{Bluher1}. 

\begin{lemma}\label{quintic}
		Let $m$ be an odd number, and let $f(x)=x^5+ax+a\in \ftwom[x]$ where $a\neq 0$. Then $f(x)$ has 0,1 or 3 roots in $\ftwom$.
\end{lemma}

\begin{proof}
	Suppose $\alpha$ is a root of  $f(x)$ in $\ftwom$. Then it follows that
	\[f(x)=(x+\alpha)(x(x+\alpha)^3 + \frac{\alpha^4}{\alpha+1}),\]
	and hence if we let $y = x + \alpha $, it suffices to prove that $g(y)=y^4 + \alpha y + \frac{\alpha^4}{\alpha+1}$ has no or 2 roots in $\ftwom$. But, if $g(y)$ has at least one root like $\beta$, then we have
	\[g(y)=y^4 + \alpha y + \frac{\alpha^4}{\alpha+1} = (y+\beta)((y+\beta)^3 + \alpha).\] 
	Now, the result follows as $x\mapsto x^3$ is a bijection over odd-degree extensions of $\ftwo$ which implies that 
	 $(y+\beta)^3 + \alpha$ has exactly one root.
\end{proof}
\fi

\subsection{Difference sets}
In this section we gather and prove some results concerning the construction of difference sets which be used in later sections. 

\subsubsection{Complementary difference sets}
The following result shows how one can get a difference set different from the given one for free simply by taking its complement.

\begin{theorem}\label{Complementary}
	Let $G$ be a cyclic group of order $v$, and let $\mathcal{D}$ be a cyclic difference set with parameter $(v,k,\lambda)$. Then 
	\begin{itemize}
	
\item[(i)] 	$G\setminus\mathcal{D}$ is a cyclic difference set with parameters  $(v,v-k,v-2k+\lambda)$,

\item[(ii)] if  $k$ is an integer coprime  with $v$, then $\mathcal{D}^k$ is a difference set with parameters
$(v,k,\lambda)$, and  $G\setminus\mathcal{D}^k$ is a cyclic difference set with parameters  $(v,v-k,v-2k+\lambda)$.
	
\end{itemize}
\end{theorem}

Notice that if $\mathcal{D}$ is a difference set with Singer parameters $(2^m-1,2^{m-1}-1,2^{m-2}-1)$, then 
$G\setminus\mathcal{D}$ is a difference set with Singer parameters $(2^m-1,2^{m-1},2^{m-2})$.

The following is an immediate corollary of the above theorem and the fact that the set of elements of trace one is a difference set with Singer parameters $(2^m-1,2^{m-1},2^{m-2})$.

\begin{cor} \label{Trace-DiffSet}
	Let as before $T_n$ denote the elements of $\ftwom$ whose trace is one. Furthermore, let $\gcd(n, 2^m-1)=1$. Then $T_n$ is a difference set in $\ftwomstar$ with Singer parameters  $(2^m-1,2^{m-1},2^{m-2})$.
\end{cor}

\subsubsection{Difference sets from sextic binomials over binary fields}
It is well-known that the polynomial $f(x)=x^6+x$ induces a two to one map over the odd-degree extensions of $\ftwo$, and $D(f)^*$ is a difference set with Singer parameter $(2^m-1,2^{m-1}-1,2^{m-2}-1)$ in $\ftwomstar$~\cite{Maschietti, Dillon-Dobbertin}. Another sextic polynomial with the same properties is $g(x)=x^6+x^5$. This fact can be deduced from the results of Glynn~\cite[Result 8]{Glynn} and Maschietti~\cite{Maschietti} (see also~\cite[Page 352]{Dillon-Dobbertin}).  Although this fact is probably well-known to the experts, it is not written explicitly anywhere in the literature. Since we will need these facts in the later sections, we record them as a lemma, and give the proof of some of the claims about $g(x)$ as some details of the proof  will be required later.

 \iffalse
 we know that if $f(x)=x^6+x$ or $f(x)=x^6+x^5$, then the set 
\[
D=\ftwom\setminus D(f)
\]
is a difference set in the multiplicative subgroup of $\ftwom$ whenever $m$ is an odd number.  Furthermore, from the same results it follows that both $x^6+x$ and $x^6+x^5$ induce a two to one map over odd-degree extensions of $\ftwo$.   
\fi

\begin{lemma}\label{sextic-equation}
	Let $m$ be an odd number, $f(x)=x^6+x$, $g(x)=x^6+x^5$, and $a,b,c\in\ftwom$.  Then 
	\begin{itemize}
	
	\item[(i)] $f$ and $g$ induce two to one maps over $\ftwom$,
	
	\item[(ii)] if  $a,b\in\ftwom$ and $g(a)=g(b)$, then $\Tr(a+b)=1$,
	
	\item[(iii)]  the sets $D(f)^*$ and $ D(g)^*$ both are difference sets with singer parameters $(2^m-1,2^{m-1}-1,2^{m-2}-1)$ in $\ftwomstar$.
	
	\end{itemize}
\end{lemma}
\begin{proof}
	For the proof of the claims about $f(x)$ see~\cite{Maschietti}, and for the proof of claim (iii) about $g(x)$ see the results of Glynn~\cite[Result 8]{Glynn} and Maschietti~\cite{Maschietti}. To prove (i) for $g(x)$, suppose $y\neq x$ and  
	\begin{equation}\label{xy-equation}
	x^6+x^5=y^6+y^5.
	\end{equation}

	If we let $y=\frac{x}{z}$  in this  equation and solve for $x$, we get
	\[
	x=\frac{z^6+z}{z^6+1}=1+\frac{1}{z^5+z^4+z^3+z^2+z+1}.
	\]
	But $z^5+z^4+z^3+z^2+z+1=D_5(z+1)$  where $D_5(z)=z^5+z^3+z$ is  the Dickson polynomial  of degree five which is a permutation polynomial over odd degree extensions of $\ftwo$ (see~\cite[Page 353]{Dillon-Dobbertin}). Thus  $z^5+z^4+z^3+z^2+z+1$ is a permutation polynomial, and hence for given $x\in\ftwom - \{1\} $, there is a unique $z\in \ftwom - \{0\}$ satisfying the above equation and hence there is a unique $y$ satisfying~\eqref{xy-equation}. Furthermore,  obviously, if $x=1$, $y\neq x$,  and $g(y)=g(1)$, then $y=0$. So, $g(x)$ induces a two to one map over the odd-degree extensions of $\ftwo$.
	
	To prove (ii), from the proof of (i) we first notice that if $\{a,b\}\neq \{0,1\}$, then for some $z\neq 1$ we have 
	\[
	\{a,b\}=\left\{\frac{1+z^5}{1+z^6},z\frac{1+z^5}{1+z^6}\right\}.
	\]
	Using this fact and the fact that $\Tr(1)=1$ over the odd-degree extensions of $\ftwo$,  we have
	\begin{eqnarray*}
	\Tr(a+b)&=&\Tr(1+\frac{z^5+z}{z^6+1})=1+\Tr(\frac{z(z+1)^4}{(z+1)^2(z^4+z^2+1)})=1+\Tr(\frac{z^3+z}{z^4+z^2+1}).\\
	\end{eqnarray*}
	But
	\[
	\frac{z^3+z}{z^4+z^2+1}=\frac{z}{z^2+z+1}+\frac{z^2}{z^4+z^2+1}=\frac{z}{z^2+z+1}+(\frac{z}{z^2+z+1})^2.
	\]
	As $\Tr(u^2+u)=0$ for every $u\in\ftwom$, from the above equation we conclude that 
	\[
	\Tr(\frac{z^3+z}{z^4+z^2+1})=0,
	\]
	and hence $\Tr(a+b)=1$. Now, as the claim is trivial when $\{a,b\}\neq \{0,1\}$, we are done with the proof of (ii).

\end{proof}
\subsubsection{Dillon-Dobbertin difference sets}
We will need the following result originally conjectured by Dobbertin~\cite{Dobbertin}. 
\begin{theorem}~\cite{Dillon-Dobbertin}\label{Dillon-Dobbertin}
	Let $m$ be an odd number, and $k$ be an integer. Furthermore, let  $1\le k\le \frac{m}{2}$, $(k,m)=1$, $d=2^{2k}-2^k+1 $, and $f(x)=(x+1)^d+x^d+1\in \ftwom[x]$.  Then $D(f)^*$ is a difference set in $\ftwomstar$ with singer parameters $(2^m-1, 2^{m-1}-1, 2^{m-2}-1)$.
\end{theorem}

\subsection{Factorization of trinomials over binary fields}

\subsubsection{Quadratic and Cubic Trinomials}
The following lemma is a direct consequence of the fact that $\Tr(x^2)=\Tr(x)$ for every $x\in\ftwom$.
\begin{lemma}\label{Quad-zeros}
The quadratic polynomial $f(x)=x^2+x+a\in \ftwom[x]$ has a solution in $\ftwom$ if and only if $\Tr(a)=0$. 
\end{lemma}

The following lemma is a special case of a much more general result of Berlekamp~\cite{Berlekamp}.
\begin{lemma}\label{Cubic-Berlekamp}
	Suppose $f(x)=x^3+px+q\in \ftwom[x]$ where $q\neq 0$. Then $f(x)$ has an odd number of irreducible factors over $\ftwom$ if and only if $\Tr(\frac{p^3}{q^2}+1)=0$. 
\end{lemma}

\iffalse
The following is an immediate corollary of the above lemma.
\begin{cor}\label{Cubic-cor-Berlekamp}
	Let $m$ be an odd number, and let $f(x)=x^3+x+a\in \ftwom[x]$ where $a\neq 0$. If $f$ is irreducible over $\ftwom$, then so is $x^2+x+a^{-1}$. 
\end{cor}
\fi

\subsubsection{$f(x)=x^{2^k+1}+ax+b$}
Let $k<m$ be a positive integer, and let 
$$f(x)=x^{2^k+1}+ax+b$$ 
be a polynomial in $\ftwom[x]$. This polynomials appears in many different contexts and applications, including, difference sets, finite geometry, the computation of cross correlation between $m$-sequences, Cohen-Matthews polynomials and index-calculus algorithms for solving discrete logarithm problem (see for example~\cite{Bluher1, Robert, Dillon-Dobbertin}), and hence has been studied extensively in the literature.  As this polynomial will appear in several places in the later sections, we gather and prove some results about it. 

To start with, following~\cite{Dillon-Dobbertin}, we need to  make some notations as follows:

\[A_1(x)=x,\]
\[A_2(x)=x^{2^k+1},\]
\[A_{i+2}(x)=x^{2^{(i+1)k}}A_{i+1}(x)+x^{2^{(i+1)k}-2^{ik}}A_i(x), i \ge 1,\]
\[B_1(x)=0,\]
\[B_2(x)=x^{2^k-1},\]
\[B_{i+2}(x)=x^{2^{(i+1)k}}B_{i+1}(x)+x^{2^{(i+1)k}-2^{ik}}B_i(x), i \ge 1.\]

Also for each positive integer $k^{'}<m$, let:
\begin{equation}\label{DefinitionOfR}
R_{k,k^{'}}(x)=(\sum_{i=1}^{k^{'}}A_i(x))+B_{k^{'}}(x)
\end{equation}
and 
\begin{eqnarray}\label{DefinitionOfQ}
Q_{k,k^{'}}(x)=\frac{\sum_{i=1}^{k^{'}}x^{2^{ik}}}{x^{2^k+1}}, k^{'} odd, \nonumber\\
Q_{k,k^{'}}(x)=\frac{1+\sum_{i=1}^{k^{'}}x^{2^{ik}}}{x^{2^k+1}}, k^{'} even.
\end{eqnarray}
The following theorems taken from~\cite{Helleseth-1} and~\cite{Helleseth-2}, respectively, are essentially the special cases of a result due to Bluher~\cite{Bluher1}. 
\begin{theorem}\label{HKLemma-0}
   Let $k$ be a positive integer such that $gcd(k,m)=1$ and $k^{'}$ be the multiplicative inverse of $k$ modulo $m$. Let $p(x)=x^{2^k+1}+x$ and  $P_i$ be the set of all $a \ne 0$ in $\ftwom$ where $|P_p(a)|=i$.  Then the function $p(x)$ is a $(0,1,3)$-map and we have:
   
 \begin{itemize}
	\item $|P_0|=\frac{2^{m}+1}{3}$ and for all $a \in P_0$ we have $\Tr (R_{k,k^{'}}(a^{-1}))=1$.
	
	\item $|P_1|=2^{m-1}$ and $a \in P_1$ if and only if $\Tr (R_{k,k^{'}}(a^{-1}))=0$.
	
	\item $|P_3|=\frac{2^{m-1}-1}{3}$ and for all $a \in P_3$ we have $\Tr (R_{k,k^{'}}(a^{-1}))=1$.
\end{itemize}
\end{theorem}

\begin{theorem}~\cite[Theorem 1]{Helleseth-2}\label{root-counting}
	Let $m$ be and odd number, $a\in\ftwom$, and $$f_a(x)=x^{2^k+1}+ax+a.$$
	Also, let $d=\gcd(k,m)$. Then $f_a(x)=0$ has $0,1,2$ or $2^d+1$ roots in $\ftwom$.  Furthermore, if $Z_i$ denotes the set of $a\in\ftwom$ for which $f_a(x)=0$ has exactly $i$ roots in $\ftwom$, then:
	\begin{itemize}
		\item[(i)] $|Z_0|=\frac{(2^m+1)2^{d-1}}{2^d+1}$,
		\item[(ii)] $|Z_1|=2^{m-1}$
		\item[(iii)] $|Z_2|=\frac{(2^m-1)(2^{d-1}-1)}{2^d-1}$,
		\item[(iv)] $|Z_{2^d+1}|=\frac{2^{m-d}-1}{2^{2d}-1}.$
	\end{itemize}
	
\end{theorem}

\begin{lemma}\label{Six-Five-Polys}
	Let $m$ be an odd number, $a\in\ftwom$, $a\neq 0$, $f_a(x)=x^5+ax+a$, and $g_a(x)=x^6+x^5+a$. Then,  $g_a(x)=0$ has exactly two solutions in $\ftwom$ if and only if $f_a(x)=0$ has exactly one solution in $\ftwom$. Furthermore, If $f_a(x)=0$ has three or no solutions, then $g_a(x)=0$ has no solution.
	
\end{lemma}
\begin{proof}
First we prove that if $g_a(x)=0$ has two solutions in $\ftwom$, then $f_a(x)=0$ has exactly one solution in $\ftwom$. If $a=0$, then the claim is trivial. So, suppose $a\neq 0$, and suppose that $g_a(x)=0$ has two solutions  in $\ftwom$. Notice that from the previous lemma we know that $g_a(x)=0$ has either zero or two solutions in $\ftwom$. Now, on the one hand from the proof of the previous lemma we see that for a unique $z\in\ftwom$
\begin{equation}
a=(\frac{1+z^5}{1+z^6})^6+(\frac{1+z^5}{1+z^6})^5,
\end{equation}
and on the other hand it is easy to check that $f_a(\alpha)=0$ where
\[
\alpha=\frac{ z^5 + z^4 + z^3 + z^2 + z}{z^6 + z^5 + z^3 + z + 1}.
\]
Thus, whenever $g_a(x)=0$ has two roots, then $f_a(x)=0$ has at least one root. To show that $f_a(x)=0$ has exactly one root, we give an indirect proof. To start with, we notice that if $u$ is a solution of $f_a(x)=0$, and $v$ is a solution of $g_a(x)=0$, then
\begin{equation}\label{sextic-equation}
\frac{u^5}{u+1}=v^6+v^5=a,
\end{equation}
and hence $(u,v)$ is a point on the curve $C$ given by
\[
C:\;\;x^5+(x+1)(y^6+y^5)=0. 
\]
Now, let 
\[
A=\left\{(a,u,v)|a,u,v\in \ftwom,\frac{u^5}{u+1}=v^6+v^5=a \right\},
\]
and 
\[
A_a=\left\{(u,v)|u,v\in \ftwom,\frac{u^5}{u+1}=v^6+v^5=a \right\}.
\]
Then using the facts that $x^6+x^5$ induces a two to one map by Lemma~\ref{sextic-equation} and the fact that for given $a$ whenever the equation $g_a(x)=0$ has two solutions, then $f_a(x)$ has at least one solution, we deduce that:
\begin{itemize}
\item[(i)] $|A_a|\ge2$ for $2^{m-1}$ values of $a$, and
 
 \item[(ii)]  $|A_a|=0$ for the rest of the values of $a$. 
\end{itemize}  
Hence
$$
|A|=\sum_{a\in\ftwom}|A_a|\ge 2^{m-1}.2=2^m. 
$$
From these facts and the fact that $\#C=|A|$, we see that in order to prove our claim, it suffices to prove that there are exactly $2^m$ points on the curve $C$ since this would imply that if for some $a\in\ftwom$, $|A_a|\neq 0$, then $|A_a|=2$. So, in the rest of the prove we show that $\#C=2^m$.

Now, let $E$ be the curve given by 
\[
E:\;\;x+(x+1)(y^2+y)=0,
\]
 and let $\zeta$ be the rational map given by
\begin{eqnarray*}
	\zeta:& C&\longrightarrow E\\\nonumber
	&(x,y)&\longrightarrow (\frac{x}{y},\frac{xy}{x+y}+(\frac{x}{y})^2).\nonumber
\end{eqnarray*}

It follows that $\zeta$ is well defined map (birational map) with the inverse map $\zeta^{-1}$ given by
\begin{eqnarray*}
	\zeta^{-1}:& E&\longrightarrow C\\\nonumber
	&(x,y)&\longrightarrow ((x+1)(y+x^2),\frac{(x+1)(y+x^2)}{x}).\nonumber
\end{eqnarray*}
From the definition of $\zeta$ and $\zeta^{-1}$, we see that 
$\zeta^{-1}\circ \zeta$ and $\zeta\circ \zeta^{-1}$ are identity maps on $C-{(0,0)}$ and $E-{(0,0)}$, respectively, and hence $\#C(\ftwom)=\#E(\ftwom)$. But rewriting the equation of $E$ as
\[
E:\;\; y^2+y=\frac{x}{x+1},
\]
for $x\neq 1$ and using Lemma~\ref{Quad-zeros}, we have $\#E(\ftwom)=2^m$. Notice that $x\mapsto \frac{x}{x+1}$ is a bijection of $\ftwom\setminus \{1\}$ and for $2^{m-1}$ values of $x$ we have $\Tr(\frac{x}{x+1})=0$. This finishes the proof of one direction.

Now, to prove the converse, let $G$ and be the set of all values of $a$ for which $g_a(x)=0$ has two solutions, and similarly $F$ be the set of all values of $a$ for which $f_a(x)=0$ has exactly one solution. By the fact that $x^6+x^5$ induces a two to one map, we have $|G|=2^{m-1}$, and by Theorem~\ref{root-counting} we have $|F|=2^{m-1}$, and furthermore  from the converse direction we have $G\subset F$. Thus $F=G$, and hence whenever $f_a(x)=0$ has one solution, then $g_a(x)=0$ has exactly two solutions. The rest of the claims are trivial.

\end{proof}

%###########################################################################################
%\section{Difference sets from  trinomials}\label{Trinomials}

\section{The value-set equivalence of some of the trinomials}\label{Value-set-equivalence}
In this section, we show that in order to prove Theorem~\ref{Ding-conj-diff}, it suffices to prove the claim of Theorem~\ref{Ding-conj-diff} for four rational functions. First, we need a definition.

\iffalse
~\ref{Value-set} to include all functions defined over $\ftwomstar$.

Let $\ftwo[x,x^{-1}]$ denote the polynomial ring generated by $x$ and $x^{-1}$ over $\ftwo$. Then any element $f(x)=c+\sum b_i x^{-i}+\sum d_jx^j$ induces a map from $\ftwom$ to itself if we let $f(a)=c+\sum b_i a^{-i}+\sum d_j a^j $ for $a\neq 0$ and $f(0)=c$. We can define an equivalence relation between the elements of $\ftwo[x,x^{-1}]$ with constant term zero as follows. Roughly speaking in the following definition two functions are equivalent if they have the same image.
\begin{definition}
Let $f$ and $g$ be in $\ftwo[x,x^{-1}]$ such that $f(0)=g(0)=0$. We call $f$ and $g$  $\ftwom$- equivalent and write $f\sim g$ if  there exists some permutation $\psi$  of  $\ftwomstar$  $\psi:\ftwom\longrightarrow \ftwom$ with $\psi(0)=0$ such that $f(a)=g(\psi(a))$ for every $a\in\ftwom$.
\end{definition}

\begin{definition}\label{Value-set-function}
Given a function $f:\ftwomstar\longrightarrow \ftwomstar$, the set
	\[D{{(f)}}^*=\{f(x):x\in \ftwomstar\}\setminus\{0\},\]is called the punctured value-set of $f$ and its size is denoted by $|D(f)^*|$.
\end{definition}

Notice that since all the polynomials appearing in Theorem~\ref{Ding-conj-diff} vanish at zero, $D(f)^*$ obtained using  this new definition for a polynomial $f$ appearing in the theorem coincides with  $D(f)^*$ obtained from Definition~\ref{Value-set}. Thus, from now on, we will work with this new definition.
\fi
\begin{definition}
	Let $f$ and $g$ be two functions over $\ftwomstar$. We call $f$ and $g$, {\it{value-set equivalent}} and write $f\sim g$ if  $D(f)^*=D(g)^*$.
\end{definition}	

Obviously, $\sim$ is an equivalence relation, and if $f\sim g$ and $D(f)^*$ is a difference set with singer parameters in $\ftwomstar$, then so is $D(g)^*$.  Hence in order to prove Theorem~\ref{Ding-conj-diff}, it suffices to identify the value-set equivalence classes of the functions appearing in Theorem~\ref{Ding-conj-diff} and prove the theorem for the representative of each class. The following theorem shows that the set of polynomials appearing in the conjecture can be partitioned into at most four equivalence classes. 

\begin{theorem}\label{PartitionPolys}
Let $m$ be an odd number, and let $f_1,\ldots,f_{11}$ be the polynomials appearing in Theorem~\ref{Ding-conj-diff}. Furthermore, let $\sigma=2^{\frac{m+1}{2}}$. Then
\begin{itemize}
	
\item[(a)] $f_5\sim f_8\sim f_{11}\sim x^{-4}+x^6+x$.	
	
\item[(b)] $f_1\sim f_2\sim {{x}^{3}}+{{x}^{20}}+{{x}^{-48}}$, 

\item[(c)] $f_4\sim f_6\sim f_9\sim x^{-\frac{\sigma}{2}}+x^{-\frac{\sigma - 1}{2}}+x $, and

\item[(d)] $f_3\sim f_7\sim f_{10}\sim x^{3\sigma+4}+x^{-2}+x $.

 \end{itemize}
\end{theorem}
\begin{proof} First, we notice that if $f$ and $g$ are two functions over $\ftwomstar$ and there
	 exists some permutation $\psi:\ftwomstar\longrightarrow \ftwomstar$ such that $f(a)=g(\psi(a))$ for every $a\in\ftwomstar$, then $f\sim g$. Now, if $t$ is an integer number such that $\gcd(t,2^m-1)=1$, then $\psi(x)=x^t$  is a permutation of $\ftwom$. From this fact we deduce that
$\phi:x\rightarrow x^6$ is a permutation of $\ftwom$ when $m$ is an odd number, and 	
\[
f_5(x)\sim f_5(x^6)=x^{5.2^m-4}+x^{4.2^m-8}+x.
\]
But this in turn, as $x^{2^m}=x$ in $\ftwom$, implies that
\[
 f_5\sim f_5(x^6)\sim x+x^{-4}+x^6.
\]
 Similarly, again, using  $x^{2^m}=x$ in $\ftwom$, we have 
\[
f_8\sim f_8(x^{-4})\sim x+x^{-4}+x^6,
\]
and 
\[
f_{11}\sim x+x^{-4}+x^6,
\]
 from which we deduce part (a).
 
In a completely similar fashion we have

$$f_1(x)\sim {{f}_{1}}({{x}^{3}})\sim x^{3.2^m-51}+x^{2^m+19}+x^3\sim {{x}^{-48}}+{{x}^{20}}+{{x}^{3}},$$ and

	$$ f_2(x)\sim {{f}_{2}}({{x}^{-48}})\sim {{x}^{3}}+{{x}^{20}}+{{x}^{-48}}$$
which proves part (b). Finally, the remaining parts follow from the following relations:

\begin{itemize}

	\item[(c)] $f_4(x)\sim f_6(x^{-\frac{\sigma}{2}})\sim f_9(x^{-\frac{\sigma-1}{2}})\sim x^{-\frac{\sigma}{2}}+x^{-\frac{\sigma - 1}{2}}+x$,
	
	\item[(d)] $f_3(x)\sim f_7(x^{3\sigma+4})\sim f_{10}(x^{-2})\sim x^{3\sigma+4}+x^{-2}+x.$

\end{itemize}

\end{proof}

\section{The proof of  the main result}\label{Main-result-proof}
Using Theorem~\ref{PartitionPolys}, in order to prove Theorem~\ref{Ding-conj-diff}, it suffices to prove that the four rational functions $x^{-4}+x^6+x$, ${{x}^{3}}+{{x}^{20}}+{{x}^{-48}}$, $ x^{3\sigma+4}+x^{-2}+x $ and  $x^{-\frac{\sigma}{2}}+x^{-\frac{\sigma - 1}{2}}+x $ where $\sigma=2^{\frac{m+1}{2}}$ give rise to difference sets with Singer parameters. To do so, in the following four subsections corresponding to the four rational functions, first we prove that the size of the  corresponding value-sets are $2^{m-1}$ and then prove that the corresponding value-sets are difference sets.

\subsection{\boldmath{$x^{-4}+x^6+x$}}
\begin{theorem}\label{Low-degree}
 Let $f(x)=x^{-4}+x^6+x$, and let $m$ be an odd number. Then $f$ induces a $(0,1,4)$-map over $\ftwom$, and furthermore 
 $|D(f)^*|=2^{m-1}$.
\end{theorem}
\begin{proof}
In order to prove the claim of the theorem, instead of trying to see what elements appear in $D(f)^*$, we give an indirect proof by looking at the solutions $(x,y)$ of the equation $f(x)=f(y)$. Observe that if $(x,y)$ is a solution of the equation $f(x)=f(y)$, then it is a point on the curve $C$ given by
\[
C: y^4(x^{10}+x^5+1)+x^4(y^{10}+y^5+1)=0.
\]
It is easy to see that $C$ is the union of the following three curves:
\[
C_1: x+y=0,\; C_2: x^3y^2 + x^2y^3 + 1=0, 
\]
and 
\[
C_3: x^6y^2 + x^4y^4 + x^3 + x^2y^6 + x^2y + xy^2 + y^3=0. 
\]
Obviously, all the trivial solutions of the equation $f(x)=f(y)$, i.e., the solutions where $x=y$ correspond to the points on $C_1$, and the remaining solutions correspond to the points on $C_2$ and $C_3$. The rest of the proof involves three steps. In the first step, we find the number of affine $\ftwom$-rational points on $C_2$ and $C_3$. In the second step, we prove that $f$ is a $(0,1,4)$-map, and finally in the third step, the results of the previous steps are used to obtain the size of the punctured value-set of $f$. 
 
Now, we proceed to the first step:
\begin{itemize}
    \item[(i)] $\#C_2(\ftwom)$: Let $D$ be the curve given by $y^2+y=x^5$, and let $\eta$ be the rational map given by
    \begin{eqnarray*}
    \eta:& C_2&\longrightarrow D\\\nonumber
    &(x,y)&\mapsto (xy,x^3y^2).\nonumber
    \end{eqnarray*}
    It is easy to verify that $\eta$ is a birational map with the inverse map $\eta^{-1}$ given by 
    \begin{eqnarray*}
    \eta^{-1}:& D&\longrightarrow C_2\\\nonumber
    &(x,y)&\mapsto (\frac{y}{x^2},\frac{x^3}{y}).\nonumber
    \end{eqnarray*}
   Furthermore, from the definition of $\eta$ and $\eta^{-1}$ we see that $\eta^{-1}\circ \eta$ is the identity map on the affine piece of $C_2$, and $\eta\circ\eta^{-1}$ is the identity map on the affine piece of $D$ except at the points $(0,0)$ and $(0,1)$ on which $\eta\circ \eta^{-1}$ is not defined. Thus, using Lemma~\ref{Quad-zeros} and the fact $\psi: x\mapsto x^5$ is an bijection on $\ftwom$ when $m$ is an odd number, we have
   \[
   \#C_2(\ftwom)=\#D(\ftwom)-2=2^m-2.
   \]
    \item[(ii)] $\#C_3(\ftwom)$:  Let $E_1$ be the curve given by $x^5(y^3+1)^2+y^3=0$, and let $\zeta$ be the rational map given by
    \begin{eqnarray*}
    \zeta:& C_3&\longrightarrow E_1\\\nonumber
    &(x,y)&\longrightarrow (x,\frac{x+y}{x}).\nonumber
    \end{eqnarray*}
    It follows that $\zeta$ is a birational map with the inverse map $\zeta^{-1}$ given by
    \begin{eqnarray*}
    \zeta^{-1}:& E_1&\longrightarrow C_3\\\nonumber
    &(x,y)&\mapsto (x,x(y+1)).\nonumber
    \end{eqnarray*}
    Now, from the definition of $\zeta$ and $\zeta^{-1}$, we see that 
       $\zeta^{-1}\circ \zeta$ and $\zeta\circ \zeta^{-1}$ are identity maps on $C_3-{(0,0)}$ and $E_1-{(0,0)}$, respectively. Thus, considering the fact that the point $(0,0)$ lies on both $C_3$ and $E_1$, we get $$\#C_3(\ftwom)=\#E_1(\ftwom).$$ So, we need to compute $\#E_1(\ftwom)$. In order to compute $\#E_1(\ftwom)$, we notice that the map $y\mapsto y^3$ is an automorphism of $\ftwom$ whenever $m$ is an odd number, and hence we need to compute the number of affine $\ftwom$-rational points of the curve $E_2: x^5(y+1)^2+y=0$ which is birational to $E_3: y^2+y=x^5$ with the birational map taking $(x,y)$ on $E_2$ to $(x,\frac{1}{y+1})$ on $E_3$. Again using Lemma~\ref{Quad-zeros} and the fact that the map taking $x$ to $x^5$ is an isomorphism of odd-degree extensions of $\ftwo$ and arguments similar to the previous ones we get $$\#C_3(\ftwom)=\#E_1(\ftwom)=\#E_2(\ftwom)=\#E_3(\ftwom)-1=2^m-1.$$
\end{itemize}
Having completed the first step of the proof, we proceed to the second step of the proof. We first notice that if there are $l$ points $(a,b_1),(a,b_2),\ldots,(a,b_l)$ on $C_2$ and $C_3$, then $f(a)=f(b_1)=\cdots=f(b_l)$, and hence $f(a)$ has $l+1$ preimages. Thus, in order to prove that $f$ is a $(0,1,4)$- map, we need to prove that for any $a\in\ftwom$, there are either none or in total three points on $C_2$ and $C_3$ with $x$-coordinate equal to $a$ and $y$-coordinate in $\ftwom$. To this end, examining the proof of Case (ii) of the previous step, we see that for any $a$ in $\ftwom$, the total number of points with $x$-coordinate equal to $a$ and $y$-coordinate in $\ftwom$ on $C_3$ is the same as the total number of points with $x$-coordinate equal to $a$ and $y$-coordinate in $\ftwom$ on $E_3: y^2+y=x^5$ which is two if $Tr(a^5)=0$ and zero, otherwise. So, we need to prove:
\begin{itemize}
    \item[(a)] for any $a\in \ftwom$, there is exactly one point with $x$-coordinate equal to $a$ on $C_2$ if and only if there are exactly two points on $E_3$ with $x$-coordinates equal to $a$, and 
    \item[(b)] if there are three points with $x$-coordinate equal to $a$ on $C_2$, then there is none on $E_3$.
\end{itemize}
We prove (a) and (b) by letting $y=xz$ in the equation of the curve $C_2$, which results in 
\[
E_4: x^5(z^3+z)+1=0.
\]
To prove (a), notice that there is exactly one point on $C_2$ with $x$-coordinate equal to $a$ if and only if there is exactly one point on $E_4$ with $x$-coordinate equal to $a$. But using Lemma~\ref{Cubic-Berlekamp} there is exactly one point on $E_4$ with $x$-coordinate equal to $a$ if and only if $Tr(a^{10})=Tr(a^{5})=0$ which is true if and only if there are exactly two points with $x$-coordinate equal to $a$ on $E_3$. Case (b) can be proved similarly using Lemma~\ref{Cubic-Berlekamp}. This finishes the proof of the second step. 

Now, in order to obtain the size of $D(f)^*$, we define an equivalence relation between the elements of $\ftwom$. We say two elements $a$ and $b$ are equivalent and write $a\equiv b$ in $\ftwom$ if $f(a)=f(b)$. Obviously, $\equiv$ is an equivalence relation, and finding the size of the $D(f)^*$ is equivalent to finding the number of the equivalence classes of $\equiv$. In the second step we proved that $f$ is a $(0,1,4)$-map. Thus every equivalence class of $\equiv$ has either one or four elements. We see that each equivalence classes of size four like $\{a_1,a_2,a_3,a_4\}$ contributes four points to the set of points on $C_1$, namely $(a_1,a_1),(a_2,a_2),(a_3,a_3),(a_4,a_4)$, and twelve points to the set of points which lie on either $C_2$ and $C_3$. So, if the number of equivalence classes of size one and four are denoted by $r$ and $s$, respectively, then since we have to exclude the point $(0,0)$ on $C_3$ from our count we get that
\[
12s=(\#C_2(\ftwom))+(\#C_3(\ftwom)-1)=2^{m+1}-4,
\]
and hence $s=\frac{2^{m-1}-1}{3}$. Finally, since $r=2^m-4s$, we find that $r+s$, the total number of equivalence classes, is 
\[
r+s=2^m-3s=2^{m-1}+1,
\]
and hence excluding the equivalence class corresponding to zero we get $|D(f)^*|=2^{m-1}$.

\end{proof}

\begin{theorem}\label{EndCase1}
 Let $f(x)=x^{-4}+x^6+x$, and let $m$ be an odd number. Then $D(f)^*$ is a difference set with Singer parameters $(2^m-1,2^{m-1},2^{m-2})$ in $\ftwomstar$.
\end{theorem}
\begin{proof}
Let $g(x)=x^6+x$. Then by Lemma~\ref{sextic-equation},  $D(g)^*$ is a difference set with Singer parameters $(2^m-1,2^{m-1}-1,2^{m-2}-1)$ in $\ftwom^*$. Thus using Theorem~\ref{Complementary} and  the fact that $|D(f)^*|=2^{m-1}$ by the previous theorem, in order to prove that $D(f)^*$ is a difference set in $\ftwomstar$ with Singer parameter $(2^m-1,2^{m-1},2^{m-2})$, it suffices to prove that $D(f)^*\cap D(g)^*=\emptyset$, or equivalently the value-sets of $f$ and $g$ partition $\ftwomstar$. In order to prove this claim, we need to show that the equation
\[
x^{-4}+x^6+x=y^6+y
\]
has no solution $(x,y)$ over $\ftwomstar$ or equivalently the following curve has no non-trivial point over odd-degree extensions of $\ftwo$:
\[
C_1:\;\; x^{10}+x^5+1+x^4(y^6+y)=0.
\]
But we have
\[
x^{10}+x^5+1+x^4(y^6+y)=(x^5+x^4y)^2+(x^5+x^4y)(x^3y^2+x^2y^3+1)+(x^3y^2+x^2y^3+1)^2.
\]
So we can apply Lemma~\ref{No-point} to the curve $C_1$ . Now, it is easy to see that there is no finite double point on $C_1$ over odd-degree extensions of $\ftwo$, and hence there is no point over $\ftwom$.  This finishes the proof .
\end{proof}

    \begin{remark}
    One may wonder that how the curves $D$ and $E$ have been obtained. We have obtained the curve $E$ by blowing up the curve $C_2$ at the origin.
    \end{remark}
    
 \subsection{\boldmath{$x^{-48}+x^{20}+x^3$}}
% Like the proof of the previous case, we first find the size of the punctured value-set and then use it to %prove that the  punctured value set of $f(x)=x^{-48}+x^{20}+x^3$ is a difference set.
 \begin{theorem}\label{HIgh-degree}
 	Let $f(x)=x^{-48}+x^{20}+x^3$, and let $m$ be an odd number. Then $f$ induces a $(0,1,4)$-map over $\ftwom$, and furthermore 
 	$|D(f)^*|=2^{m-1}$.
 \end{theorem}
 \begin{proof}
 Similar to the previous case, the proof involves three steps; point-counting step for the related curves, proving that $f$ induces a $(0,1,4)$-map, and finally obtaining the size of the punctured value-set.  Considering the equation $f(x)=f(y)$, we get the curve
 $$
 C:  y^{48}(x^{58}+x^{51}+1)+x^{48}(y^{58}+y^{51}+1)=0
 $$
 which can be written as the union the following four curves:
 \[
 C_1: x+y=0,\;\;
C_2:  x^{12}y^{12}+x^4y^3+x^3y^4+x^7+y^7=0,
 \]

 \[
 C_3: x^{12}y^{12}(x+y)^3+(x^2+xy+y^2)^5=0,
 \]
 and
 \iffalse
 \[
 C_4:  x^{40}y^{24} + x^{36}y^{28} + x^{32}y^{32} + x^{31}y^{16} + x^{30} + x^{28}y^{36} +
 x^{28}y^{19} + x^{24}y^{40} + x^{24}y^6 + x^{19}y^{28} + x^{16}y^{31} + x^6y^{24} +y^{30}=0
 \]
 \fi
 
 \[
 C_4: (x^{20}y^{12}+y^{20}x^{12}+(x^3+y^3)^5)^2+(x^{16}y^{16})^2+x^{16}y^{16}(x^{20}y^{12}+y^{20}x^{12}+(x^3+y^3)^5)=0.
 \]
 So, at the first step of the proof we find the number of affine rational points on the curves $C_2, C_3$ and $C_4$ as follows:
 
 \begin{itemize}
 	\item[(i)]  $\#C_2(\ftwom)$:  Let $D$ be the curve given by $(y^2+y)^6+(y^2+y)^5=x^{17}$, and let $\eta$ be the rational map given by 
 	\begin{eqnarray*}
 		\eta:& C_2&\longrightarrow D\\\nonumber
 		&(x,y)&\mapsto (x,\frac{x}{y}).\nonumber
 	\end{eqnarray*}
 	It is easy to verify that $\eta$ is a birational map with the inverse map $\eta^{-1}$ given by 
 	\begin{eqnarray*}
 		\eta^{-1}:& D&\longrightarrow C_2\\\nonumber
 		&(x,y)&\mapsto (x,\frac{x}{y}).\nonumber
 	\end{eqnarray*}
 	
 	Examining the definitions of $\eta$ and $\eta^{-1}$, we see that $\eta^{-1}\circ\eta$ and $\eta\circ\eta^{-1}$ are identity maps on $C_2-(0,0)$ and $D-(0,0)$, respectively. Thus
 	\[
 	\#C_2(\ftwom)=\#D(\ftwom).
 	\]
 	
 	Now, let $E$ be the curve given by $y^6+y^5=x^{17}$, and $\mu: D\longrightarrow E$ be the map given by
 	\begin{eqnarray*}
 		\mu:& D&\longrightarrow E\\\nonumber
 		&(x,y)&\mapsto (x,y^2+y).\nonumber
 	\end{eqnarray*}
  From the definition of $\mu$, we see that the point $(a,b)$ on $E$ has no preimage on $\ftwom$ whenever $\Tr(b)=1$, and it has exactly two preimages whenever $\Tr(b)=0$. So
  \[
  \#D(\ftwom)=2|\{(a,b)\in E: \Tr(b)=0\}|.
\]
 But, by Lemma~\ref{sextic-equation}, whenever $m$ is an odd number, for half the points $(a,b)$ on $E$ we have $\Tr(b)=0$. Thus 
 \[
 \#D(\ftwom)=\#E(\ftwom).
 \]
 Finally, since the map $x\mapsto x^{17}$ is an automorphism of $\ftwom$, and  $y^6+y^5$ is a two to one map on $\ftwom$ by Lemma~\ref{sextic-equation} , we have   $\#E(\ftwom)=2^m$, and hence 
 \[
  \#C_2(\ftwom)=\#E(\ftwom)=2^m.
 \]
 	\item[(ii)] $\#C_3(\ftwom)$:  Let $E_1$ be the curve given by $x^{17}y^3+(1+y^3)^5=0$, and let $\zeta$ be the rational map given by
 	\begin{eqnarray*}
 		\zeta:& C_3&\longrightarrow E_1\\\nonumber
 		&(x,y)&\longrightarrow (x,\frac{x+y}{y}).\nonumber
 	\end{eqnarray*}
 	It follows that $\zeta$ is a birational map with the inverse map $\zeta^{-1}$ given by
 	\begin{eqnarray*}
 		\zeta^{-1}:& E_1&\longrightarrow C_3\\\nonumber
 		&(x,y)&\mapsto (x,x/(y+1)).\nonumber
 	\end{eqnarray*}
 	Now, from the definition of $\zeta$ and $\zeta^{-1}$, we see that 
 	$\zeta^{-1}\circ \zeta$ and $\zeta\circ \zeta^{-1}$ are identity maps on $C_3-{(0,0)}$ and $E_1-{(0,1)}$, respectively. Thus, considering the fact that the point $(0,0)$ lies on $C_3$ and $(0,1)$ lies on $E_1$, we get $$\#C_3(\ftwom)=\#E_1(\ftwom).$$ So, we need to compute $\#E_1(\ftwom)$. In order to compute $\#E_1(\ftwom)$, we notice that the map $y\mapsto y^3$ is an automorphism of $\ftwom$ whenever $m$ is an odd number, and hence we need to compute the number of affine $\ftwom$-rational points of the curve $E_2: x^{17}y+(y+1)^5=0$ which is birational to $E_3: y^5+x^{17}(y+1)=0$ with the birational map taking $(x,y)$ on $E_2$ to $(x,y+1)$ on $E_3$. But the map $\phi:x\mapsto x^{17}$ is an isomorphism of $\ftwom$ whenever $m$ is an odd number. Thus we can apply Lemma~\ref{root-counting} to count the number of points on $E_3$. Using that lemma there are $2^{m-1}$ values of $x$ for which there is a unique $y$ such that $(x,y)$ is a point on $E_3$, and $\frac{2^{m-1}-1}{3}$ values of $x$ for which there are exactly three values of $y$ such that $(x,y)$ is a point on $E_3$. Hence in total there are $2^{m-1}+3\frac{2^{m-1}-1}{3}=2^m-1$ points on the curve $E_3$. Using this fact and arguments similar to the previous ones we get $$\#C_3(\ftwom)=\#E_1(\ftwom)=\#E_2(\ftwom)+1=\#E_3(\ftwom)+1=2^m.$$
 	
 	\item[(iii)]  $\#C_4(\ftwom)$: Applying Lemma~\ref{No-point}, we need to check that on $C_4$ there is no finite double  point over odd degree extensions of $\ftwo$ which is very easy to verify. So, $\#C_4(\ftwom)=0$ whenever $m$ is an odd number. This completes the first step of the proof in this case.
 \end{itemize}
 	 Similar to the proof of the previous case (change to subsection), in order to prove that $f$ is a $(0,1,4)$- map, we need to prove that for any $a\in\ftwom$, there are either none or in total three points on $C_2$ and $C_3$ with $x$-coordinate equal to $a$ and $y$-coordinate in $\ftwom$. To this end, examining the proof of Case (ii) of the previous step, we see that for any $a$ in $\ftwom$, the total number of points with $x$-coordinate equal to $a$ and $y$-coordinate in $\ftwom$ on $C_2$ is the same as the total number of points with $x$-coordinate equal to $a$ and $y$-coordinate in $\ftwom$ on $E: y^6+y^5=x^{17}$ which is  either two or zero by Lemma~\ref{sextic-equation}. We also see that  the total number of points with $x$-coordinate equal to $a$ and $y$-coordinate in $\ftwom$ on $C_3$ is the same as the total number of points with $x$-coordinate equal to $a$ and $y$-coordinate in $\ftwom$ on $E_2: x^{17}y+(y+1)^5$.  So, we need to prove:
\begin{itemize}
	\item[(a)] for any $a\in \ftwom$, there are exactly two point with $x$-coordinate equal to $a$ on $E$ if and only if there are exactly one points on $E_2$ with $x$-coordinates equal to $a$, and 
	\item[(b)] if there are three points with $x$-coordinate equal to $a$ on $E_2$, then there is none on $E$.
\end{itemize}
But, both of the above claims follow from Lemma~\ref{Six-Five-Polys}. Hence,  $f(x)$ induces a $(0,1,4)$-map over $\ftwom$. This finishes the second step of the proof. 

The proof of the third step is completely similar to the proof of the third step of the proof of Theorem~\ref{Low-degree}. 
\end{proof}

\begin{theorem}
	Let $f(x)=x^{-48}+x^{20}+x^3$, and let $m$ be an odd number. Then $D(f)^*$ is a difference set with Singer parameters $(2^m-1,2^{m-1},2^{m-2})$ in $\ftwomstar$.
\end{theorem}
\begin{proof}
	Let $g(x)=(x+1)^{241}+x^{241}+1$. Then by Lemma~\ref{Dillon-Dobbertin},  $D(g)^*$ is a difference set with Singer parameters $(2^m-1,2^{m-1}-1,2^{m-2}-1)$ in $\ftwomstar$. Thus using Theorem~\ref{Complementary} and  the fact that $|D(f)^*|=2^{m-1}$ by the previous theorem, in order to prove that $D(f)^*$ is a difference set in $\ftwomstar$ with Singer parameter $(2^m-1,2^{m-1},2^{m-2})$, it suffices to prove that $(D(f)^*)^{17}\cap D(g)^*=\emptyset$ or equivalently $(D(f)^*)^{17}$ and $ D(g)^*$ partition $\ftwomstar$. In order to prove this claim, we need to show that the equation
	\[
(x^{-48}+x^{20}+x^3)^{17}=(y+1)^{241}+y^{241}+1
	\]
	has no solution $(x,y)$ over $\ftwomstar$ or equivalently as $\psi:x\mapsto 1/x$ is an isomorphism of $\ftwomstar$, the following equation has no solution  $(x,y)$ over $\ftwomstar$
	\[
	(x^{48}+x^{-20}+x^{-3})^{17}=(y+1)^{241}+y^{241}+1.
	\]
	
		\[
	(x^{68}+x^{17}+1)^{17}=x^{340}((y+1)^{241}+y^{241}+1).
	\]
	As the map $x\mapsto x^{17}$ is an isomorphism over odd-degree extensions of $\ftwo$, if we set $u=x^{17}$, then  the latter claim is equivalent to the claim that the following curve has no non-trivial point over the odd-degree extensions of $\ftwo$:
	\[
	C:\;\;(u^4+u+1)^{17}+u^{20}((y+1)^{241}+y^{241}+1)=0.
	\]
	Now, it is easy to see that there is no finite double point on $C$ over odd-degree extensions of $\ftwo$, and furthermore, using MAGMA computer algebra system we see that 
	\[
p(u,y)=(u^4+u+1)^{17}+u^{20}((y+1)^{241}+y^{241}+1)
	\]
is irreducible over $\ftwo$ and reducible over $\FF_4$. Thus applying Corollary~\ref{NoPoint-Cor}, there is no finite point on $C$ over the odd-degree extensions of $\ftwo$. This completes the proof.
\end{proof}

\subsection{\boldmath{$x^{-\frac{\sigma}{2}}+x^{-\frac{\sigma - 1}{2}}+x$}}\

As with the previous cases, in this case we first obtain the size of the corresponding value set and prove that the function $f(x)=x^{-\frac{\sigma}{2}}+x^{-\frac{\sigma - 1}{2}}+x$ is semi-regular.
\begin{theorem}\label{EndCase3}
	Let $m$ be an odd integer, $\sigma=2^{\frac{m+1}{2}}$,  and $f(x)=x^{-\frac{\sigma}{2}}+x^{-\frac{\sigma - 1}{2}}+x$. Then $f$ is a $(0,1,4)$-map on $\ftwomstar$, $(D(f)^{*})^{\sigma+1} = T_1$, and hence $|D(f)^{*}|=2^{m-1}$.
\end{theorem}
The proof of  the above theorem is completely different from the proofs of Theorems~\ref{Low-degree} and ~\ref{HIgh-degree} in the previous cases. In order to prove the above theorem, we need some preparations. First, we start with some notations specific to this and the next subsections.

\begin{notation}\label{DefinitionOfFunctions}
Let $m$ be an odd integer,  and $\sigma=2^{\frac{m+1}{2}}$. Then we let:
\begin{itemize}
%\item	$f(x)=x^{-\frac{\sigma}{2}}+x^{-\frac{\sigma - 1}{2}}+x$
\item $g(x)= x^{\sigma} + x$,
\item $h(x) = x + x^{-1} + x^{-\sigma} + x^{\sigma -1} + x^{2 - \sigma}$,
\item $p(x)=x^{\sigma+1}+x$,
\item $R(x)=x^{\sigma+1}+x^{\sigma-1}+x$, and 
\item $Q(x)=\frac{x^{\sigma}+x^2+1}{x^{\sigma+1}}$.
\end{itemize}

\end{notation}

\begin{remark}\label{QR-Remark}
  Notice that $p(x)$ is the same $p(x)$ defined in Theorem \ref{HKLemma} for $k=\frac{m+1}{2}$.  Also, $R$ and $Q$ are respectively the same as $R_{k,k^{'}}$ and $Q_{k,k^{'}}$ introduced in \ref{DefinitionOfR} and \ref{DefinitionOfQ} for $k=\frac{m+1}{2}$ and $k^{'}=2$ which are multiplicative inverse of each other modulo $m$ . 
\end{remark}

Considering the above remark, Theorem~\ref{HKLemma-0} can be rewritten as the follows.
\begin{theorem}\label{HKLemma}
	Let $\sigma=2^{\frac{m+1}{2}}$, $p(x)=x^{\sigma+1}+x$, and for each $i$, $P_i$ denote the set of all $a \ne 0$ in $\ftwom$ for which $|P_p(a)|=i$.  Then the function $p(x)$ is a $(0,1,3)$-map and we have:
	
	\begin{itemize}
		\item[(i)] $|P_0|=\frac{2^{m}+1}{3}$ and for all $a \in P_0$ we have $\Tr (R(a^{-1}))=1$.
		
		\item[(ii)] $|P_1|=2^{m-1}$ and $a \in P_1$ if and only if $\Tr (R(a^{-1}))=0$.
		
		\item[(iii)] $|P_3|=\frac{2^{m-1}-1}{3}$ and for all $a \in P_3$ we have $\Tr (R(a^{-1}))=1$.
	\end{itemize}
	where $R(x)=x^{\sigma+1}+x^{\sigma-1}+x$.
\end{theorem}

We have the following lemma about $g(x)$.

\begin{lemma}\label{gLemma}
Let $g(x)$ be as introduced above, and  let $x\in\ftwom$. Then:
\begin{itemize}
 \item[(i)]  $\Tr (g(x))=0$, 
 \item[(ii)] $g(g(x))=x^2+x$, 
 \item[(iii)] $g(x)$ induces a two to one map over $\ftwom$, and 
 \item[(iv)] for $a\in\ftwom$, $|P_{g}(a)|= 2$ if and only if $\Tr (a)=0$ if and only if there is a unique $\alpha \in \ftwom$ such that $\Tr (\alpha)=0$ and $P_{g}(a)=\{\alpha, \alpha +1\}$.
\end{itemize}
	\end{lemma}
	
\begin{proof}
	To prove (i), we notice that $\Tr (x^\sigma) =\Tr(x)$ and hence
 \[\Tr (g(x))=\Tr (x^\sigma) + \Tr(x)=0.\]
 To prove (ii), as $x^{2^m}=x$, we have $(x^{\sigma})^{\sigma}=x^{2^{m+1}}=(x^{2^m})^2=x^2$ and hence 
 \[g(g(x))=(x^{\sigma}+x)^{\sigma}+(x^{\sigma}+x)=x^2+x^{\sigma}+x^{\sigma}+x=x^2+x.\]
 To prove (iii), first we notice that if $g(\alpha)=g(\beta)$ for some $\alpha,\beta$ in $\ftwom$, then $g(g(\alpha))=g(g(\beta))$ which is using (ii) is equivalent to $\alpha^2+\alpha=\beta^2+\beta$ implying that $\alpha+\beta=0$ or 1. Thus  each element of the value-set of $g(x)$ has at most to preimages.  As $g(x)=g(x+1)$,  we deduce that indeed $g(x)$ induces a two to one map over $\ftwom$. To prove (iv), let $a\in\ftwom$ and $|P_{g}(a)|= 2$.  Then $a=g(\beta)$ for some $\beta\in\ftwom$ which using (i) implies that $\Tr(a)=0$.  Furthermore,  using (i) and (iii), since half the elements of $\ftwom$ have trace equal to zero, the value-set of $g(x)$ is the set of the elements of $\ftwom$ which are of trace zero. Thus if $\Tr(a)=0$, then $P_{g}(a)=\{\alpha, \alpha +1\}$ for some $\alpha\in\ftwom$. But as $m$ is an odd number, $\Tr(1)=1$ and one of $\alpha$ and $\alpha+1$ is of trace zero and the other one is of trace one. The rest of the proof of (iv) is trivial.
\end{proof}

Next, we have the following functional equations.

\begin{lemma}\label{equalities}
For each nonzero $x \in \ftwom$ we have:

\begin{itemize}

   \item[(i)] $f(x)^{\sigma + 1} = g(h(x^{-\frac{\sigma +1}{2}})) + 1$, 
   
   \item[(ii)] $R(Q(x)^{-1})=x$ whenever $x\neq 0$,
   
   \item[(iii)] $h(x) = g(x^{-1} + x^{\sigma -1}) + x$, and
   
   \item[(iv)] $\Tr(h(x))=\Tr(x)$ for each $x\neq 0$.
   
\end{itemize}
\end{lemma}  

\begin{proof}
\begin{itemize}\itemindent -20pt
\item[(i)] Since $\sigma^{2}\equiv 2  \pmod{2^m-1}$, $x^{2^m}= x$ and $g(g(x))=x^2+x$ for every $x\in\ftwom$, we have:
\begin{align*}
     f(x)^{\sigma +1}&= f(x)^{\sigma}f(x)= f(x^{\sigma})f(x)\\ 
                              &=(x^{-1} + x^{\frac{\sigma}{2}-1} + x^{\sigma})(x^{-\frac{\sigma}{2}} + x^{-\frac{\sigma -1}{2}} + x) \\
                              &=(x^{-\frac{1}{2}} + x^{-1}) + (x^{\frac{\sigma +1}{2}} + x^{\sigma +1}) + (x^{-\frac{\sigma +1}{2}} + x^{-\frac{\sigma}{2} -1}) + 1\\
                              &=(x^{-\frac{1}{2}} + (x^{-\frac{1}{2}})^2) + (x^{\frac{\sigma +1}{2}} +(x^{\frac{\sigma +1}{2}})^2) + (x^{-\frac{\sigma +1}{2}} + (x^{-\frac{\sigma +1}{2}})^{\sigma}) + 1\\
                              &=g(x^{-\frac{1}{2}}) +g(x^{\frac{\sigma +1}{2}}) +g( x^{-\frac{\sigma +1}{2}}) +1\\
                              &=g(x^{-\frac{1}{2}} + (x^{-\frac{1}{2}})^{\sigma}) + g(x^{\frac{\sigma +1}{2}} +(x^{\frac{\sigma +1}{2}})^{\sigma} )+ g(x^{-\frac{\sigma +1}{2}})+1.
\end{align*}
But  $g(x)$ is a linearized polynomial. Thus 
\begin{align*}
f(x)^{\sigma +1}&=g((x^{-\frac{1}{2}} + (x^{-\frac{1}{2}})^{\sigma}) + (x^{\frac{\sigma +1}{2}} +(x^{\frac{\sigma +1}{2}})^{\sigma}) + x^{-\frac{\sigma +1}{2}}) +1\\
                         &=g(x^{-\frac{1}{2}} + x^{-\frac{\sigma}{2}} + x^{\frac{\sigma +1}{2}} +x^{\frac{\sigma}{2} +1} + x^{-\frac{\sigma +1}{2}}) +1\\
                         &=g(h(x^{-\frac{\sigma +1}{2}})) + 1.
\end{align*}

\indent\item[(ii)] This part is a special case of Theorem 8 of \cite{Dillon-Dobbertin}.

\item[(iii)] For each $x \ne 0$, using the fact that $\sigma^{2}\equiv 2  \pmod{2^m-1}$, we have:
\begin{align*}
h(x)&=(x^{-\sigma}+x^{-1})+(x^{2-\sigma}+x^{\sigma -1})+x\\
      &=(x^{-\sigma}+x^{-1})+(x^{\sigma(\sigma -1)}+x^{\sigma -1})+x\\           
      &=g(x^{-1} + x^{\sigma -1}) + x.                   
\end{align*}

\item[(iv)] By Part (iii), we have:
\begin{equation}\nonumber
	h(x) = g(x^{-1} + x^{\sigma -1}) + x,
\end{equation}
and therefore by Part (i) of Lemma \ref{gLemma}, 
\begin{equation}\label{EQforh}
	\Tr (h(x)) = \Tr(g(x^{-1} + x^{\sigma -1}) + x)=\Tr(g(x^{-1} + x^{\sigma -1}))+\Tr (x)=\Tr(x).
\end{equation}
\end{itemize}
\end{proof}

\begin{lemma}\label{HK_0-equality}
Let $m\ge 3$, and $a\neq 0$ be an element of $\ftwom$ such that $\Tr(a)=0$. Furthermore, for every $x\in\ftwom$ different from $0$ and $1$ let $k_0(x)=g(x)^{-1}+g(x)^{\sigma -1}+x$. Then
	\[
	|P_h(a)|=|P_{k_0}(b)|
	\]
	where $a=g(b)$.
\end{lemma}
\begin{proof}
Notice that $k_0(x)$ is well-defined as $g(x) \ne 0$ for each $x \ne 0,1$ in $\ftwom$, and the existence of $b$ is guaranteed by Part (iv) of Lemma~\ref{gLemma}. First, we prove the claim when both of the sets $P_h(a)$ and $P_{k_0}(b)$ are non-empty. In this case, the claim  would follow if we prove that 
$l_0=g|_{P_{k_0}(b)}$ is a bijection between the two sets $P_h(a)$ and $P_{k_0}(b)$.
In order to see that $l_0$ is well-defined, we notice that by Part~(iii) of Lemma ~\ref{equalities} and  the fact that $g(x)$ is a linearized polynomial, we have
\begin{equation}\label{hK_0-equation}
	g(k_0(x))=g(g(x)^{-1}+g(x)^{\sigma-1})+g(x)=h(g(x)).
\end{equation}
Thus, if $x_0 \in P_{k_0}(b)$, then
\[h(g(x_0))=g(k_0(x_0))=g(b)=a,\] 
and therefore $g(x_0) \in P_h(a)$. In order to prove that $l_0$ is an injection,
let $l_0(u)=l_0(v)$ and hence  $g(u)=g(v)$ for $u,v \in P_{k_0}(b)$. Then by the definition of $k_0$ and the fact that $k_0(u)=k_0(v)=b$ we get:
\[u=k_0(u)+g(u)^{-1}+g(u)^{\sigma-1}=b+g(u)^{-1}+g(u)^{\sigma-1}=k_0(v)+g(v)^{-1}+g(v)^{\sigma-1}=v.\]
So $l_0$ is injective.

Next, we prove that $l_0$ is surjective. Let $x_0$ be an element of $P_h(a)$. Then, by Part (iv) of Lemma~\ref{equalities}, \[\Tr (x_0)=\Tr(h(x_0))=\Tr(a)=0,\] and so by Part (iv) of Lemma \ref{gLemma}, there is some $y_0$ such that $x_0=g(y_0)=g(y_0+1)$. It suffices to prove that either $y_0$ or $y_0+1$ is  in $P_{k_0}(b).$

For $t \in \{y_0, y_0+1\}$, using~\eqref{hK_0-equation}, we have:
\[g(k_0(t))=h(g(t))=h(x_0)=a=g(b),\]
and hence by Lemma \ref{gLemma}, we must have
$k_0(t)=b$ or $b+1$. As $g(t+1)=(t+1)^{\sigma}+(t+1)=(t)^{\sigma}+1+(t)+1=g(t)$, it follows that $k_0(t+1)=k_0(t)+1$, and therefore either $y_0$ or $y_0+1$ is in $P_{k_0}(b)$. This finishes the proof when both $P_h(a)$ and $P_{k_0}(b)$ are non-empty. If both $P_h(a)$ and $P_{k_0}(b)$  are empty, there is nothing to prove. Now, suppose $P_h(a)$ is empty and $P_{k_0}(b)$ is not. Then  for $x_0 \in P_{k_0}(b)$, as above, we have
\[h(g(x_0))=g(k(x_0))=g(b)=a,\] 
and therefore $g(x_0) \in P_h(a)$ which is a contradiction. If $P_h(a)$ is not empty and $P_{k_0}(b)$ is, then our proof of surjectivity  of $l_0$ above shows that indeed $P_{k_0}(b)$ is not empty which is contradiction again.
\end{proof}

\begin{lemma}\label{QK_0-equality}
Let $m\ge 3$, and $a\neq 0$ be an element of  $\ftwom$ such that $\Tr(a)=0$. Furthermore, for every $x\in\ftwom$ different from $0$ and $1$ let $k_0(x)=g(x)^{-1}+g(x)^{\sigma -1}+x$. Then
\[
|P_{k_0}(b)|=|P_p(Q(a^{\frac{\sigma}{2}}+1))|
\]
where $a=g(b)$.
\end{lemma}

\begin{proof}
	Since $x\in P_{k_0}(b)$ if and only if
	\[g(x)^{-1}+g(x)^{\sigma-1}+x=b,\]
	 multiplying both sides of the equation by $g(x)=x+x^{\sigma}$, we deduce that $x\in P_{k_0}(b)$ if and only if
	\[1+(x+x^{\sigma})^{\sigma}+x(x+x^{\sigma})=b(x+x^{\sigma})\]
	or equivalently
	\[x^{\sigma+1} + (b+1)x^{\sigma} + bx +1 =0.\]
	Thus, if we let
	\[
	A^0_b=\left\{x\in \ftwom|x^{\sigma+1} + (b+1)x^{\sigma} + bx+1 =0\right\},
	\]
	then
	\[
	|P_{k_0}(b)|=|A^0_b|.
	\]
	Now, let  $\psi$ be the map
	\begin{align*}
		\psi:&~ A^0_b\longrightarrow P_p(Q(a^{\frac{\sigma}{2}}+1))\\
		& ~x\longrightarrow  \frac{x+b+1}{b^{\frac{\sigma}{2}}+b+1}.
	\end{align*}
	To see that  $\psi$ is a well-defined map, we notice that $b^{\frac{\sigma}{2}}+b+1 \ne 0$ since $\Tr(b^{\frac{\sigma}{2}}+b+1)=1$, and furthermore 
	if $x\in A^0_b$ and we let 
	\[
	y=\frac{x+b+1}{b^{\frac{\sigma}{2}}+b+1},
	\]
	then by a straightforward calculation
	\[p(y)=y^{\sigma+1} + y= \frac{b^2+b+1}{{(b^{\frac{\sigma}{2}}+b+1)}^{\sigma+1}}=Q(b^{\frac{\sigma}{2}}+b+1)=Q(g(b)^{\frac{\sigma}{2}}+1)=Q(a^{\frac{\sigma}{2}}+1).\]
	Obviously, $\psi$ is a bijection. Thus
	$|P_{k_0}(b)|=|P_p(Q(a^{\frac{\sigma}{2}}+1))|$.
\end{proof}
%%%%%%%%%%%%%%%%Trace 1%%%%%%%%%%%%%%%%%%%%%%%%%%%%%%%%%
\begin{lemma}\label{HK_1-equality}
	Let $m\ge 3$, and $a$ be an element of $\ftwom$ such that $\Tr(a)=1$. Furthermore, for every $x\in\ftwom$ different from $0$ and $1$ let $g_1(x)=g(x)+1$ and $k_1(x)=(g_1(x))^{-1}+(g_1(x))^{\sigma -1}+x$. Then $|P_h(1)|=1$, and for $a\neq 1$ we have
	\[
	|P_h(a)|=|P_{k_1}(b)|
	\]
	where $a=g_1(b)$, that is, $a+1=g(b)$.
\end{lemma}

\begin{proof}
	Notice that since $\Tr(g(x))=0$ and $\Tr(1)=1$ over the odd-degree extensions of $\ftwo$, $g(x) \ne 1$ for each $x $ in $\ftwom$ and hence $k_1(x)$ is well-defined. Furthermore, the existence of $b$ is guaranteed by Part (iv) of Lemma~\ref{gLemma} and the fact that $\Tr(a+1)=0$. The proof is similar to the proof of Lemma~\ref{HK_0-equality}. 
	In this case, the claim  would follow if we prove that 
	$l_1=g_1|_{P_{k_1}(b)}$ is a bijection between $P_h(a)$ and $P_{k_1}(b)$.
	In order to see that $l_1$ is well-defined, we notice that by the fact that $g(x)$ is a linearized polynomial and  Part~(iii) of Lemma ~\ref{equalities}, we have
	\begin{align}\label{hK_1-equation}
		\begin{split}
			g_1(k_1(x))&=g(k_1(x)) + 1\\
			&=g(g_1(x)^{-1}+g_1(x)^{\sigma-1}+x) + 1\\
			&=g(g_1(x)^{-1}+g_1(x)^{\sigma-1})+g(x) + 1\\
			&=g(g_1(x)^{-1}+g_1(x)^{\sigma-1})+g_1(x)\\
			&=h(g_1(x)),
		\end{split}            
	\end{align}

	which implies that if $x_0 \in P_{k_1}(b)$, then
	\[h(g_1(x_0))=g_1(k_1(x_0))=g_1(b)=a.\] 
The rest of the proof is similar to the proof of Lemma~\ref{HK_0-equality}.		
\end{proof}

\begin{lemma}\label{QK_1-equality}
	Let $m\ge 3$, and $a$ be an element of $\ftwom$ such that $\Tr(a)=1$. Furthermore, for every $x\in\ftwom$ different from $0$ and $1$ let $g_1(x)=g(x)+1$ and $k_1(x)=(g_1(x))^{-1}+(g_1(x))^{\sigma -1}+x$. Then:
	\begin{itemize}
	\item [(i)] $|P_h(1)|=1$, and

\item [(ii)] if $a\neq 1$, then $|P_{k_1}(b)|=|P_p(Q(a^{\frac{\sigma}{2}}+1))|$
where $a=g_1(b)$, that is, $a+1=g(b)$.
	\end{itemize}
\end{lemma}
\begin{proof}
	Since $x\in P_{k_1}(b)$ if and only if
	\[(g(x)+1)^{-1}+(g(x)+1)^{\sigma-1}+x=b,\]
	multiplying both sides of the equation by $g(x)+1=x+x^{\sigma}+1$, we deduce that $x\in P_{k_1}(b)$ if and only if
	\[1+(x+x^{\sigma}+1)^{\sigma}+x(x+x^{\sigma}+1)=b(x+x^{\sigma}+1)\]
	or equivalently
	\[x^{\sigma+1} + (b+1)x^{\sigma} +(b+1)x +b =0.\]
	Thus, if we let
	\[
	A^1_b=\left\{x\in \ftwom|x^{\sigma+1} + (b+1)x^{\sigma} + (b+1)x+b =0\right\},
	\]
	then
	\[
	|P_{k_1}(b)|=|A^1_b|.
	\]
	Now, in order to prove (i), notice that when $a=1$, then $b=0$ or $1$, and thus using 
 the previous lemma $|P_h(1)|=|P_{k_1}(1)|=|A^1_1|$.   As $\gcd(\sigma+1,2^m-1)=1$, $A^1_1=\{1\}$  and hence, $|P_h(1)|=1$.
 
  In order to prove (ii), let  $\psi$ be the map
	\begin{align*}
		\psi:&~ A^1_b\longrightarrow P_p(Q(a^{\frac{\sigma}{2}}+1))\\
		& ~x\longrightarrow  \frac{x+b+1}{b^{\frac{\sigma}{2}}+b}.
	\end{align*}
	The map  $\psi$ is a well-defined map since
	if $x\in A^1_b$ and we let 
	\[
	y=\frac{x+b+1}{b^{\frac{\sigma}{2}}+b},
	\]
	then by a straightforward calculation
	\[p(y)=y^{\sigma+1} + y= \frac{b^2+b+1}{{(b^{\frac{\sigma}{2}}+b)}^{\sigma+1}}=Q(b^{\frac{\sigma}{2}}+b)=Q(g(b)^{\frac{\sigma}{2}})=Q(a^{\frac{\sigma}{2}}+1),\]
	 and furthermore $b^{\frac{\sigma}{2}}+b\ne 0$ whenever $b\neq 0,1$ which is the case when $a\neq 1$.
	Obviously, $\psi$ is a bijection. Thus
	$|P_{k_1}(b)|=|P_p(Q(a^{\frac{\sigma}{2}}+1))|$.
\end{proof}
%%%%%%%%%%%%%%%%%%%%%%%%%%%%%%%%%%%%%%%%%%%%%%%%%%%%%%%%%
The next theorem which is immediate from Lemmas~\ref{HK_0-equality},~\ref{QK_0-equality},~\ref{HK_1-equality} and ~\ref{QK_1-equality} relates the number of preimages of $a$ and $Q(a^{\frac{\sigma}{2}}+1)$ under two different maps.
\begin{theorem}\label{ChangeOfFunctions}
Let $a\neq 0$ be an element of $\ftwom$ where $m\ge 3$. Then:
\begin{itemize}
\item[(i)] $|P_h(1)|=1$,  and 

\item[(ii)] $|P_h(a)|=|P_p(Q(a^{\frac{\sigma}{2}}+1))|$ when $a\neq 1$. 

\end{itemize}
\end{theorem}
The next theorem shows that the function $h(x)$ is a $(0,1,3)$-map over odd-degree extensions of $\ftwo$.
\begin{theorem}\label{hLemma}
The function  $h(x)$ is a $(0,1,3)$-map. Furthermore, if we denote by $H_i$ the set of all $a$ in $\ftwomstar$ for which $|P_h(a)|=i$, then 
\begin{itemize}
	\item[(i)] $|H_0|=\frac{2^{m}+1}{3}$,
	
	\item[(ii)] $|H_1|=2^{m-1}$,
	
	\item[(iii)] $|H_3|=\frac{2^{m-1}-1}{3}$, and  
	
	\item[(iv)]  $H_1=T_1$ .
\end{itemize}
\end{theorem}

\begin{proof}
For each non-negative integer $i$, let 
\begin{align*}
	A_i=\{a \in \ftwomstar \setminus \{1\}|~\;\; |P_p(Q(a^{\frac{\sigma}{2}}+1))|=i\},\; B_i=\{a \in \ftwomstar |~\;\; |P_p(Q(a))|=i\},
\end{align*}
and as in Theorem~\ref{HKLemma}, let
\[
P_i=\{a \in \ftwomstar |~\;\; |P_p(a)|=i\}.
\]
 If $i\neq 1$, using Theorem~\ref{ChangeOfFunctions}, we have $|H_i|=|A_i|$. Furthermore,  since  $\psi: a \mapsto a^{\frac{\sigma}{2}}+1$ is a bijection on $\ftwomstar \setminus \{1\}$, we have $|A_i|=|B_i|$ for $i\neq 1$. But as $R(Q(x)^{-1})=x$ by  Lemma~\ref{equalities}, $Q(x)$ induces a bijection on $\ftwomstar$, and hence $|B_i|=|P_i|$ for $i\neq 1$. Thus $|A_i|=|P_i|$ for $i\neq 1$
from which, using  Theorems~\ref{HKLemma} and ~\ref{ChangeOfFunctions}, the fact that $h(x)$ is a $(0,1,3)$-map and Parts (i) and (iii) follow.

If $i=1$, then using Theorem \ref{ChangeOfFunctions}  we see that an element $a\neq 1$ is  in $H_1$ if and only if $Q(a^{\frac{\sigma}{2}}+1) \in P_1$. But  by  Part (ii) of Theorem~\ref{HKLemma} and Part (ii) of  Lemma~\ref{equalities},   $Q(a^{\frac{\sigma}{2}}+1) \in P_1$ if and only if \[\Tr(R(Q(a^{\frac{\sigma}{2}}+1)^{-1}))=\Tr(a^{\frac{\sigma}{2}}+1)=0\] which is equivalent to $\Tr(a)=1$. Thus,  since $|P_h(1)|=1$ and $\Tr(1)=1$, Parts (ii) and (iv) follow.
\end{proof} 

	Now, we are ready to prove Theorem~\ref{EndCase3}.
	
	\iffalse
	\begin{theorem}\label{EndCase3}
		Let $m$ be an odd integer and $f(x)=x^{-\frac{\sigma}{2}}+x^{-\frac{\sigma - 1}{2}}+x$. Then $f$ is a $(0,1,4)$-map over $\ftwomstar$ and $|D(f)^{*}|=2^{m-1}$.
	\end{theorem}
	\fi
	\begin{proof}
		As $f(x)^{\sigma + 1} = g(h(x^{-\frac{\sigma +1}{2}})) + 1$ by  Part (i) of Lemma~\ref{equalities}, we have:
		\[P_f(a)=\left\{x \in \ftwom|g(h(x^{-\frac{\sigma +1}{2}}))+1=a^{\sigma +1}\right\}.\]
		In order to prove that $f$ is a $(0,1,4)$-map, it suffices to prove that:
		\begin{itemize}
			\item[(i)] if $\Tr(a^{\sigma +1})=0$, then the set $P_f(a)$ is empty, and
			
			\item[(ii)] if $\Tr(a^{\sigma +1})=1$, then the set $|P_f(a)|=1$ or $4$. 
			
		\end{itemize}
		To see (i), we notice that  by Part (i) of Lemma~\ref{gLemma} we have $\Tr (g(h(x^{-\frac{\sigma +1}{2}})) + 1) =1$  for all $x \in \ftwom$  which implies that $P_f(a)$ is empty whenever  $\Tr (a^{\sigma +1})=0$. 
		
		To prove (ii), let $\Tr (a^{\sigma +1})=1$. Then $\Tr (a^{\sigma +1}+1)=0$, and so by Part (iv) of Lemma \ref{gLemma}, there is a unique $\alpha$ with $\Tr (\alpha)=1$ for which $P_g(a^{\sigma +1}+1)=\{\alpha, \alpha+1\}$. Therefore, $x \in P_f(a)$ if and only if $h(x^{-\frac{\sigma +1}{2}})=\alpha$ or $\alpha + 1$, and hence, using the fact that $\psi: x \mapsto x^{-\frac{\sigma +1}{2}}$ is a bijection on $\ftwomstar$, we have
		\[|P_f(a)|=|P_h({\alpha})| + |P_h({\alpha+1})|. \]
		 The claim follows since by Theorem \ref{hLemma}, we have $|P_h({\alpha})|=1$ and $|P_h({\alpha +1})|=0$ or 3. 
		
		Finally, since $\gcd(\sigma + 1, 2^m -1)=1$, and hence the map $\psi:x \mapsto x^{\sigma +1}$ is a bijection on $\ftwom$,  the facts that  $|D(f)^{*}|=2^{m-1}$ and $(D(f)^{*})^{\sigma+1} = T_1$  follow from (ii).
		
	\end{proof}
Here is the main result of this subsection.
\begin{theorem}\label{EndCase4}
 Let $m\ge 5$ be an odd integer. If $f(x)=x^{-\frac{\sigma}{2}}+x^{-\frac{\sigma - 1}{2}}+x$, then $D(f)^{*}$ is a difference set with Singer parameters $(2^m-1,2^{m-1},2^{m-2})$ in $\ftwomstar$.
\end{theorem}

\begin{proof}
By Theorem \ref{EndCase3},  $(D(f)^{*})^{\sigma+1}= T_1$. Furthermore, by Corollary~\ref{Trace-DiffSet}, $T_1$ is a difference set with Singer parameters $(2^m-1,2^{m-1},2^{m-2})$ in $\ftwomstar$. Therefore, by Theorem \ref{Complementary}, $D(f)^{*}$ is a difference set with Singer parameters $(2^m-1,2^{m-1},2^{m-2})$ .
\end{proof}
	
\subsection{\boldmath{$x^{3\sigma+4}+x^{-2}+x$}	}\

Like the previous subsection, we start with some notations.  For the sake of convenience, we repeat some notations from the previous subsection. 
\begin{notation}
	For each nonzero $x \in \ftwom$, let:
\begin{itemize}
\item $f_s(x)=x^{3\sigma+4}+x^{-2}+x$ and
\item $h_s(x)=x^{\sigma +1} + x^{\sigma} + x^{\sigma-1} + x^{\sigma -2} + x^{ - \sigma +1}$.
\item	$f(x)=x^{-\frac{\sigma}{2}}+x^{-\frac{\sigma - 1}{2}}+x$
\item $g(x)= x^{\sigma} + x$
\item $h(x) = x + x^{-1} + x^{-\sigma} + x^{\sigma -1} + x^{2 - \sigma}$
\item $R(x)=x^{\sigma+1}+x^{\sigma-1}+x$
\end{itemize}
where $\sigma=2^{\frac{m+1}{2}}$ as before.
\end{notation}

We have the following functional equations.
\begin{lemma}\label{SecondEqualities}
For each nonzero $x \in \ftwom$ we have:

\begin{itemize}
   
   \item[(i)] $f_s(x)^3 = g(h_s(x^{3\sigma +3})) + 1$.
   
   \item[(ii)] $h_s(x)=h(R(x))$
\end{itemize}
\end{lemma} 

\begin{proof}
 We have
\begin{align*}
f_s(x)^3&=f_s(x)^2f_s(x)= f_s(x^2)f_s(x)\\
            &=(x^{6\sigma+8}+x^{-4}+x^2)(x^{3\sigma+4}+x^{-2}+x)\\
            &=(x^{9\sigma+12} + x^{6\sigma+9}) + (x^{6\sigma+6} + x^{3\sigma+6}) + (x^{3\sigma} + x^3) + (x^{-6} + x^{-3}) + 1\\
            &=((x^{6\sigma+9})^{\sigma} + x^{6\sigma+9}) + ((x^{3\sigma+6})^{\sigma} + x^{3\sigma+6}) + ((x^3)^{\sigma} + x^3) + ((x^{-3})^2 + x^{-3}) + 1\\
            &=g(x^{6\sigma+9} )+g( x^{3\sigma+6}) + g(x^3) + ((x^{-3})^2 + x^{-3}) +1
\end{align*} 
But $g(g(x))= x^2 + x$. So,
\[
(x^{-3})^2 + x^{-3}=g(x^{-3} +x^{-3\sigma}),
\]
and hence 
\[
f_s(x)^3=g(x^{6\sigma+9}) +g( x^{3\sigma+6}) +g( x^3) +g(x^{-3} +(x^{-3})^{\sigma})+1.
\]
Since $g(x)$ is a linear function, we have
 \begin{align*}
f_s(x)^3&=g(x^{6\sigma+9} + x^{3\sigma+6} + x^3 + (x^{-3} +(x^{-3})^{\sigma})) +1\\
            &=g(x^{6\sigma+9} + x^{3\sigma+6} + x^3 + x^{-3} +x^{-3\sigma}) +1\\
            &= g(h_s(x^{3\sigma+3})) + 1.\\
\end{align*} 
 This proves (i). 

To verify (ii), it suffices to multiply  both sides of the equation by  $R(x)^{\sigma+1}$ and then expand them.
\end{proof} 
Now, we can prove the main theorem of this subsection.
\begin{theorem}
 Let $m$ be an odd integer and $f_s(x)=x^{3\sigma+4}+x^{-2}+x$. Then $D(f_s)^{*}$ is a differece set with Singer parameters $(2^m-1,2^{m-1},2^{m-2})$ in $\ftwomstar$.
\end{theorem}

\begin{proof}
 By Lemma \ref{equalities}, $R(Q(x)^{-1})=x$ whenever $x\neq 0$ . Therefore, $R(x)$ is a permuatation polynomial over $\ftwom$. Using this fact, from Part (ii) of Lemma \ref{SecondEqualities}, it follows that \[D(h)^{*}=D(h_s)^{*}.\] From this fact, the fact that 
 \[f(x)^{\sigma + 1} = g(h(x^{-\frac{\sigma +1}{2}})) + 1\] by Part (i) of the Lemma  \ref{equalities}, and the fact that  \[f_s(x)^3 = g(h_s(x^{3\sigma +3})) + 1\] by  Part (ii) of Lemma \ref{SecondEqualities}, we deduce that \[(D(f)^{*})^{\sigma+1}=(D(f_s)^{*})^3.\] 
 Using Theorem \ref{EndCase3}, this implies that \[(D(f_s)^{*})^3=T_1.\] 
 Therefore byTheorem \ref{Complementary}, $D(f_s)^{*}$ is a differece set with Singer parameters $(2^m-1,2^{m-1},2^{m-2})$ in $\ftwomstar$.
\end{proof}

\section{Tri-weight codes from trinomials}\label{Codes}
A binary linear code $\mathcal{C}$ of length $n$ and dimension $k$ is a  $k$-dimensional subspace of the $n$-dimensional vector space $\mathbb{F}_2^n$ . Any vector $c$ in $\mathcal{C}$ is called a codeword, and its (Hamming) weight denoted by $wt(c)$ is the number of nonzero components of it. The Hamming distance between the two codewords $c_1$ and $c_2$ is $wt(c_1+c_2)$. The minimum Hamming distance of a binary linear code $\mathcal{C}$ is the minimum of the weights of the nonzero codewords of $\mathcal{C}$. A binary linear code $\mathcal{C}$ of length $n$, dimension $k$, and minimum Hamming distance $d$ is called an $[n,k,d]$-code.

\iffalse
A linear $[n,k,d]$-code $\mathcal{C}$ over $\mathbb{F}_2$ is a $k$-dimensional subspace of the $n$-dimensional vector space $\mathbb{F}_2^n$ with the minimum Hamming distance $d$. 
\fi
The dual code $\mathcal{C}^\perp$  of a binary linear $[n,k,d]$-code $\mathcal{C}$ is an $n-k$-dimensional subspace of  $\mathbb{F}_2^n$  given by 
\[\mathcal{C}^{\perp}=\left\{u \in \mathbb{F}_2^{n} |
 u\cdot c=0 \text{ for all }c \in \mathcal{C}\right\}\]
 where $u\cdot c$ denotes the inner product between the two vectors $u$ and $c$.
 
 The weight enumerator polynomial of a linear code is used to encode the information about the Hamming weights of the codewords of it. If $\mathcal{C}$ is a binary linear $[n,k,d]$-code, and for each $i$, $A_i$ is the number of codewords of it of weight $i$, then the weight enumerator polynomial of $\mathcal{C}$ is 
 \[
 W_{\mathcal{C}}(z)=1 +A_1z +A_2z^2 + \dots + A_nz^n.
 \]

In some applications including the construction of secret sharing schemes,  association schemes and authentication codes, it is desirable to have codes with very sparse weight-enumerator polynomial,  and hence there is an interest in the construction of these codes. 

 In one of the methods available in the literature, for the  construction of  binary linear codes with sparse weight enumerator polynomial, one  associates a linear code $\mathcal{C}_D$  of length $n$ over $\mathbb{F}_2$ to a properly chosen subset $D=\{d_1,d_2,\ldots,d_n\}$ of $\ftwom$ by
\begin{equation}\label{Code-set}
\mathcal{C}_D=\{(Tr(xd_1), Tr(xd_2),\ldots, Tr(xd_n): x\in \ftwom\}.
\end{equation}
For example, in~\cite{Ding-Discrete,DingCode,DingLuo,DingHarald}, Cun Sheng Ding successfully applied this method and obtained linear codes with very sparse weight enumerator polynomial and also made some conjectures. One of the conjectures Cun Sheng Ding made in~\cite{Ding-Discrete} is the following conjecture  whose partial resolution is the subject of this sectoin.

\begin{conj}\label{Ding-conj-codes}\cite{Ding-Discrete}
	Let $m\ge 5$ be an odd integer. Then for every $f$ given in Theorem~\ref{Ding-conj-diff}, the binary linear code $\mathcal{C}_{D(f)^*}$ is a $[2^{m-1}, m, 2^{m-2}-2^{(m-3)/2}]$-code with weight enumerator polynomial
	\[
	1+(2^{m-2}-2^{(m-3)/2})z^{2^{m-2}-2^{(m-3)/2}}+(2^{m-1}-1)z^{2^{m-2}}+(2^{m-2}+2^{(m-3)/2})z^{2^{m-2}+2^{(m-3)/2}}.
	\]
Furthermore,	The dual code of $\mathcal{C}_{D(f)^*}$ is  a $[2^{m-1},2^{m-1}-m,3]$-code.
	
\end{conj}
We need some preliminaries which are the subjects of the next subsection.
\subsection{Weight distribution of codes and weight enumerator polynomials}
\subsubsection{Pless power moment identities}
The identities in the following theorem between the weight distributions of a code and its dual are called Pless power moment identities which can be derived from Mac-Williams identities~\cite{Mac-Williams}.  

\begin{theorem}\cite{Pless-moments}\label{VeraPlessPowerMoments}
Let $\mathcal{C}$ be a binary linear code of length $n$ and dimension $k$. Furthermore,  let $A_i$ and $A_{i}^{\perp}$ be the number of code-words of weight $i$ in $\mathcal{C}$ and $\mathcal{C}^{\perp}$, respectively. Then 
\begin{itemize}
\item[(i)] $\sum_{i=0}^{n} A_i = 2^k$,
\item[(ii)] $\sum_{i=0}^{n} iA_i = 2^{k-1}(n-A_{1}^{\perp})$,
\item[(iii)] $\sum_{i=0}^{n} i^2A_i = 2^{k-2}(n(n+1)-2nA_{1}^{\perp}+2A_{2}^{\perp})$, and
\item[(iv)] $\sum_{i=0}^{n} i^3A_i = 2^{k-3}(n^2(n+3)-(3n^2+3n-2)A_{1}^{\perp}+6nA_{2}^{\perp}-6A_{3}^{\perp}).$
\end{itemize} 
\end{theorem}

Applying the above theorem to some tri-weight linear codes we have the following theorem.

\begin{theorem}\label{DualCodeParameters}
Let $m\ge 5$ be an odd integer, and $\mathcal{C}$ be a binary linear $[2^{m-1}, m, 2^{m-2}-2^{(m-3)/2}]$- code with  weight enumerator polynomial
	\[
	1+(2^{m-2}-2^{(m-3)/2})z^{2^{m-2}-2^{(m-3)/2}}+(2^{m-1}-1)z^{2^{m-2}}+(2^{m-2}+2^{(m-3)/2})z^{2^{m-2}+2^{(m-3)/2}}.
	\]
Then $\mathcal{C}^{\perp}$ is a  $[2^{m-1},2^{m-1}-m,3]$-code.
\end{theorem}

\begin{proof}
Let $A_i$ and $A_{i}^{\perp}$ be the number of codewords of weight $i$ in $\mathcal{C}$ and $\mathcal{C}^{\perp}$, respectively.
By Parts (ii) and (iii) of \ref{VeraPlessPowerMoments}, $A_{1}^{\perp}$ and $A_{2}^{\perp}$ are zero and by Part (iv) we have $A_{3}^{\perp}=\frac{2^{2m-4}-2^{m-3}}{3}>0$. So, the minimum Hamming distance of $\mathcal{C}^{\perp}$ is 3.
\end{proof}

\subsection{Weight distribution of codes constructed from Boolean functions and Walsh transform}

In this subsection, we gather and prove some results which relate the weight distribution of codes constructed from Boolean functions to their Walsh transform. First, we need some definitions.

\begin{definition}
Let $f$ be a function from $\ftwom$ to $\mathbb{F}_2$. The support $S_f$ of $f$ is the set of all the elements of $\ftwom$ which are mapped to 1, that is,
\[S_f=\{x \in \ftwom : f(x)=1\}.\]
\end{definition}
\begin{definition}
Let $f$ be a function from $\ftwom$ to $\mathbb{F}_2$. The Walsh transform of $f$ is the function $\hat{f}:\ftwom\longmapsto \mathbb{Z}$ where  for each  $w \in \ftwom$, $\hat{f}(w)$ is given by
\[\hat{f}(w)=\sum_{x \in \ftwom} (-1)^{f(x)+\Tr (wx)}.\]
\end{definition}

The following theorem relates  the Walsh transform of a function to the weight distribution of a code constructed from its support as in~\eqref{Code-set}.
\begin{theorem}~\cite{DingCode}\label{CodeMaker}
Let $f$ be a function from $\ftwom$ to $\mathbb{F}_2$. Furthermore, let $n_f$ denote the size of the support of $f$, that is, $n_f=|S_f|$. If $2n_f + \hat{f}(w)\neq 0$ for all $w \in \ftwom$, then $\mathcal{C}_{S_f}$ is a binary linear code with length $n_f$ and dimension $m$, and its weight distribution is given by the following multiset:
\[\left\{\left\{\frac{2n_f+\hat{f}(w)}{4} : w \in \ftwom\right\}\right\} \cup \{0\}.\] 
\end{theorem}

The following lemma relates the weight enumerators of codes constructed from a set and its complement.
\begin{lemma}\label{ComplementCode}
Let $A$ be a subset of $\ftwom$ with $|A|=2^{m-1}$ and  $0 \in A$ and let $B=\ftwom \backslash A$. Suppose that $\mathcal{C}_{A \backslash \{0\}}$ is a linear $[2^{m-1}-1, m, 2^{m-2}-2^{(m-3)/2}]$-code with weight enumerator
	\[
	1+(2^{m-2}+2^{(m-3)/2})z^{2^{m-2}-2^{(m-3)/2}}+(2^{m-1}-1)z^{2^{m-2}}+(2^{m-2}-2^{(m-3)/2})z^{2^{m-2}+2^{(m-3)/2}}
	.\]
	Then  $\mathcal{C}_B$ is a linear $[2^{m-1}, m, 2^{m-2}-2^{(m-3)/2}]$-code with weight enumerator
	\[
	1+(2^{m-2}-2^{(m-3)/2})z^{2^{m-2}-2^{(m-3)/2}}+(2^{m-1}-1)z^{2^{m-2}}+(2^{m-2}+2^{(m-3)/2})z^{2^{m-2}+2^{(m-3)/2}}
	.\]
\end{lemma}
\begin{proof}
For each $Y \subset \ftwom$, let $\chi_Y:\ftwom\longrightarrow \ftwo$ be the characteristic function of $Y$, that is, $\chi_Y(y)=1$ if and only if $y\in Y$. Then it is easy to check that for each $w \in \mathbb{F}_2^m$, we have $\widehat{\chi_A}(w)=-\widehat{\chi_B}(w)$.	 The rest of the proof is immediate from Theorem \ref{CodeMaker}.
\end{proof}

\subsubsection{Some known binary linear codes with sparse weight enumerator}
The following Theorem is a special case of Theorem 7 of \cite{DingCode} for $\rho=6$.
\begin{theorem}\label{Segre}
	Let
	\[D_6=\{x^6+x:x \in \ftwom\} \backslash \{0\}.\]
	Then $C_{D_6}$ is a linear $[2^{m-1}-1, m, 2^{m-2}-2^{(m-3)/2}]$-code with weight enumerator
	\[	1+(2^{m-2}+2^{(m-3)/2})z^{2^{m-2}-2^{(m-3)/2}}+(2^{m-1}-1)z^{2^{m-2}}+(2^{m-2}-2^{(m-3)/2})z^{2^{m-2}+2^{(m-3)/2}}
	.\]
\end{theorem} 

For each positive integer number $n$, let
\[T_n=\{x \in \ftwom| \Tr(x^n)=1\}.\]
The following Theorem is a special case of Corollary 7 of \cite{Ding-Discrete}.

\begin{theorem}\label{PowersOf2DifferenceSets}
	Let $m$ be an odd integer, and $k$ be a positive integer with $gcd(k,m)=1$. Then $T_{2^k+1}$ is a linear $[2^{m-1}-1, m, 2^{m-2}-2^{(m-3)/2}]$-code with weight enumerator
	\[	1+(2^{m-2}-2^{(m-3)/2})z^{2^{m-2}-2^{(m-3)/2}}+(2^{m-1}-1)z^{2^{m-2}}+(2^{m-2}+2^{(m-3)/2})z^{2^{m-2}+2^{(m-3)/2}}
	.\]
\end{theorem}

\subsection{Partial resolution of Conjecture \ref{Ding-conj-codes}}
Here is the main result of this section.
\begin{theorem}
	Let $m\ge 5$ be an odd integer. Then for every $f$ appearing below , the binary linear code $\mathcal{C}_{D(f)^*}$ is a $[2^{m-1}, m, 2^{m-2}-2^{(m-3)/2}]$-code with weight enumerator polynomial
	\[
	1+(2^{m-2}-2^{(m-3)/2})z^{2^{m-2}-2^{(m-3)/2}}+(2^{m-1}-1)z^{2^{m-2}}+(2^{m-2}+2^{(m-3)/2})z^{2^{m-2}+2^{(m-3)/2}}.
	\]
	Furthermore,	The dual code of $\mathcal{C}_{D(f)^*}$ is  a $[2^{m-1},2^{m-1}-m,3]$-code.
	\begin{itemize}
		\item[(c)] $f_3(x)= x^{2^m-3}+x^{2^{(m+3)/2}+2^{(m+1)/2}+4}+x$.
		\item[(d)]$f_4(x)= x^{2^{m}-2^{(m-1)/2}-1}+x^{2^{(m-1)}-2^{(m-1)/2}}+x$.
		\item[(e)]$f_5(x)=x^{2^m-2-(2^{m-1}-2^2)/3}+x^{2^m-2^2-(2^m-2^3)/3}+x$.
		\item[(f)] $f_6(x)=x^{2^m-2^{(m+1)/2}+2^{(m-1)/2}}+x^{2^m-2^{(m+1)/2}-1}+x$.
		\item[(g)]$f_7(x)=x^{2^m-3(2^{(m+1)/2}-1)}+x^{2^{(m+1)/2}+2^{(m-1)/2}-2}+x$.
		\item[(h)]$f_8(x)=x^{2^m-2^{m-2}-1}+x^{2^{m-1}-2}+x$.
		\item[(i)] $f_9(x)= x^{2^{m}-2^{(m+3)/2}-3}+x^{2^{(m+1)/2}+2}+x$.
		\item[(j)]$f_{10}(x)=x^{2^m-3(2^{(m-1)/2}+1)}+x^{2^{m-1}-1}+x$.
		\item[(k)]$f_{11}(x)=x^{2^m-5}+x^6+x$.
	\end{itemize}	
	
\end{theorem}
\begin{proof}
Using  Theorem \ref{DualCodeParameters}, the claims of the theorem for dual codes would follow if we prove the claims of the theorem for the codes $\mathcal{C}_{D(f_3)^*},\ldots,\mathcal{C}_{D(f_{11})^*} $. Thus,  to prove the claims of the theorem for the the codes $\mathcal{C}_{D(f_3)^*},\ldots,\mathcal{C}_{D(f_{11})^*} $. Like Section 4, using Theorem~\ref{PartitionPolys}, to prove the theorem, it suffices to prove that the three rational functions $x^{-4}+x^6+x$, $ x^{3\sigma+4}+x^{-2}+x $ and  $x^{-\frac{\sigma}{2}}+x^{-\frac{\sigma - 1}{2}}+x $ where $\sigma=2^{\frac{m+1}{2}}$, give rise to binary linear $[2^{m-1}, m, 2^{m-2}-2^{(m-3)/2}]$-codes.

Now, let $f(x)=x^{-4}+x^6+x$. If we let $A=D_6 \cup \{0\}$ where 
\[D_6=\{x^6+x:x \in \ftwom\} \backslash \{0\}\]
as in Theorem~\ref{Segre}, then it follows from the proof of Theorem \ref{EndCase1} that  $D(f)^*$ is the complement of $A$ in $\ftwomstar$. Since $|A|=|D(f)^*|=2^{m-1}$,  the claim follows  using  Lemma ~\ref{ComplementCode} and Theorem~ \ref{Segre}.  So, the proof in this case is complete. Finally, since $D(x^{-\frac{\sigma}{2}}+x^{-\frac{\sigma - 1}{2}}+x)^{*}=T_{\sigma+1}$ and $D(x^{3\sigma+4}+x^{-2}+x)^{*}=T_3$ from the proofs of Theorems~\ref{EndCase3} and~\ref{EndCase4}, respectively, the proofs of the remaining cases follow  by applying Theorem~\ref{PowersOf2DifferenceSets} for $k=\frac{m+1}{2}$ and $k=1$.
\end{proof}
\section{Concluding Remarks}\label{Conclusions}
While trying to prove Theorem~\ref{Ding-conj-diff}, we had some observations. One observation that is worth mentioning is that if a polynomial $f(x)$ appears in the theorem, then  $f(x)/x+1$ is a two to one map whose value-set is a difference set with Singer parameters. To be more specific, the polynomials $f_i(x)/x+1$  for $i=1,2,\ldots,11$ are value-set equivalent to the following polynomials, respectively.
\begin{itemize}
	\item[(a)] $x^{-3}+x$. 
	\item[(b)] $x^4+x^3$.
	\item [(c)]$x^{-\sigma-1}+x$.
	\item[(d)] $x^{-3\sigma-4}+x$.
	\item[(e)] $x^2+x$.
	\item[(f)] $x^{\sigma+2}+x$.
	\item[(g)] $x^\sigma+x$.
	\item[(h)]$x^2+x$.
	\item [(i)] $x^{-\sigma-1}+x$.
	\item[(j)] $x^{\sigma+2}+x$.
	\item [(k)] $x+x^{-1}$.
\end{itemize}
\iffalse
To prove that the functions  $x^{-3}+x$ and $x^4+x^3$ are two to one maps, it suffices to use Berlekamp's discriminant. $x^{-3}+x$ is a two to one map when the extension degree is an odd number, and it is a $(1,4)$-map when the extension degree is an even number.\fi
We could not use the above observation towards a proof of Theorem~\ref{Ding-conj-diff}. It would be nice if the observation is used to prove Theorem~\ref{Ding-conj-diff} or it is used to obtain trinomials other than the ones appearing in the theorem whose value-sets are difference sets.
  
Finally, it should be mentioned that in order to prove the claim of Conjecture~\ref{Ding-conj-codes} for the polynomial $f_1(x)$ and $f_2(x)$ appearing in Theorem~\ref{Ding-conj-diff}, it suffices to prove Conjecture 34 of ~\cite{Ding-Discrete}.

\section*{Acknowledgements}

During the preparation of this paper we benefited from the free service provided by MAGMA computer algebra system.

%*************************************************************************

\addresseshere
\end{document}